\documentclass[12pt,reqno]{amsart}
\usepackage{mathrsfs}
\usepackage{url}
\usepackage{mathtools}
\usepackage{latexsym,epsfig,amssymb,amsmath,amsthm,color,url,bm}
\usepackage[inline,shortlabels]{enumitem}
\usepackage{hyperref}
\usepackage{amsmath,amsbsy}
\RequirePackage[numbers]{natbib}
\usepackage{mathptmx}
\usepackage[text={16cm,24cm}]{geometry}
\usepackage[capitalize,nameinlink,noabbrev,nosort]{cleveref}

\allowdisplaybreaks 
\numberwithin{equation}{section}
\newtheorem{theorem}{Theorem}[section]
\newtheorem{prop}[theorem]{Proposition}
\newtheorem{lemma}[theorem]{Lemma}
\newtheorem{corollary}[theorem]{Corollary}

\theoremstyle{definition}

\newtheorem{example}[theorem]{Example}

\theoremstyle{definition}

\DeclareMathOperator{\Prob}{\mathbb{P}}
\DeclareMathOperator{\Pos}{Pos}
\newcommand{\pca}{$m$-NED}

\title[An exactly solvable evaporation-deposition PCA]
{An exactly solvable evaporation-deposition PCA with long-distance interactions}
\date{\today}
\author{Arvind Ayyer}
\address{Arvind Ayyer, Indian Institute of Science, Bangalore 560012, Karnataka, India.}
\email{arvind@iisc.ac.in}

\author{Moumanti Podder}
\address{Moumanti Podder, Indian Institute of Science Education and Research (IISER) Pune, Dr.\ Homi Bhabha Road, Pashan, Pune 411008, Maharashtra, India.}
\email{moumanti@iiserpune.ac.in}

\begin{document}

\begin{abstract}
We consider a probabilistic cellular automaton (PCA) of evaporation-deposition on the one-dimensional lattice having $n$ sites with periodic boundary conditions, in which each site, during each epoch, can be in one of two states: $0$ and $1$. 
Fix a positive integer $m\geqslant 2$. There are two types of transitions at each discrete time, which are as follows:  
\begin{enumerate*}
\item the first site in every contiguous block of $m$ $0$s becomes a $1$ with probability $p_1$, and
\item the first site in every contiguous block of $(m-1)$ $0$s followed immediately by a $1$ also becomes a $1$ with probability $(1-p_2)$. 
\end{enumerate*}
As in a PCA, all of these transitions occur simultaneously. We show that the resulting discrete-time Markov chain is ergodic, and we give an explicit formula for its limiting distribution, the partition function and the density. We also propose necessary and sufficient conditions for this Markov chain to be reversible. 
For $m=2$, we provide a fully analytical expression for the free energy of this model.
\end{abstract}



\maketitle

\section{Introduction}\label{sec:intro}

Probabilistic Cellular Automata (PCAs) may be interpreted as discrete-time Markov chains that are generalizations of Cellular Automata (CAs), obtained by incorporating \emph{random perturbations} or \emph{random noise} into CAs (see, for instance, \cite{marcovici-sablik-taati-2019}). 
PCAs appear in many different contexts, but our interest in studying them arises from the intimate connection they have with \emph{Gibbs potentials} and \emph{Gibbs measures} in statistical mechanics~\cite{Gibbs, stat_mech_PCA, stat_mech_PCA_2, spin_models_PCA, mean_field_PCA}, and with combinatorial models such as \emph{directed animals}~\cite{dhar-1983, bousquetmelou-1998, leborgne-marckert-2007, albenque-2009}. We refer the reader to \cite{PCA_survey_old}, and the more recent \cite{louis_nardi} and \cite{mairesse-marcovici-2014}, for a detailed survey on how the theory of PCAs has developed over the years.

One of the primary motivations propelling the current work stems from the study of \emph{directed animals} in combinatorics, and the study of \emph{crystal growth models} and \emph{random gas models} in statistical mechanics, through the lens of probabilistic cellular automata. 
Dhar~\cite{dhar-1982,dhar-1983} studied the problem of enumeration of directed animals by analyzing a certain PCA on $n$ sites with periodic 
boundary conditions depending on a neighbourhood of size $2$.  
If $2$ contiguous sites are empty, then a particle deposits on the leftmost site with probability $p$. 
If at least one of $2$ contiguous sites is occupied at a given time, the leftmost one becomes empty at the next time with probability $1$.

In this work, we introduce a new PCA we call the \emph{$m$-neighbourhood 
evaporation-deposition model} in which we allow a relaxation of the hard constraint above. We work on a finite one-dimensional lattice of $n$ sites with periodic 
boundary conditions.
If $m$ contiguous sites are empty, then a particle deposits on the leftmost site with probability $p_1$. Moreover, if $m-1$ contiguous sites are empty and the $m$'th is occupied, a particle deposits on the leftmost site with probability 
$1 - p_2$. All particles evaporate with probability $1$ at the next time step. 

Closely reminiscent but distinct from our model are the more standard
evaporation-deposition models with $k$-mers~\cite{barma-grynberg-stinchcombe-1993,stinchcombe-barma-grynberg-1993,dhar-menon-1995}.
We also note the resemblance of our model with \emph{adsorption-desorption} models such as the Ziff--Gulari--Barshad model~\cite{ziff1986kinetic} and the Dickman--Burschka model~\cite{dickman-burschka-1988}.

The plan of our paper is as follows. We first define the model precisely in \cref{sec:model}. We state our main results in \cref{sec:main_results}, which include a product formula for the stationary distribution in \cref{thm:main_1},
a formula for the normalizing constant (also known as the partition function) in
\cref{thm:main_2}, a formula for the density in \cref{thm:limiting_prob_cell_occupied_by_1}, and explicit expressions for the partition function, the free energy and the density, when $m=2$, in \cref{thm:main_3}, \cref{thm:free energy m=2} and \cref{thm:density m=2} respectively.

\section{Description of the model}\label{sec:model}
We first describe the $m$-neighbourhood evaporation-deposition model, or \pca{} for short, precisely. We let $\mathbb{N}$ denote the set of all positive integers, and $\mathbb{N}_{0}$ the set of all non-negative integers. We let $[n]$ denote the set $\{1,2,\ldots,n\}$ for each $n \in \mathbb{N}$, and $[0]$ equals the empty set.

In general a \emph{probabilistic cellular automaton} (PCA) consists of: 
\begin{enumerate}
\item the \emph{universe}, each of whose elements is referred to as a \emph{cell} or a \emph{site},
\item the \emph{alphabet}, 
\item the \emph{neighbourhood-marking set},
\item the \emph{state space}, each of whose elements is referred to as a \emph{configuration},
\item and certain \emph{local, stochastic update rules}.
\end{enumerate} 
In our setup of the \pca{} model, the universe is $[n]$, 
the alphabet is $\mathcal{A}=\{0,1\}$,
the neighbour\-hood-marking set is $\mathcal{N}=\{0,1,\ldots,m-1\}$, where $m \leqslant n$,
the state space is $\Omega=\mathcal{A}^{[n]}$,
and the local, stochastic update rules are described later in \eqref{rule_1}, \eqref{rule_2} and \eqref{rule_3}. 

The sites are arranged in a ring geometry and therefore sites $n$ and $1$ are adjacent.
An important abuse of notation that happens throughout this paper, and that is to be borne in mind, is as follows: fixing $n\in\mathbb{N}$ and $i,j\in[n]$, whenever we refer to the cell indexed by $(i+j)$, we actually mean the cell indexed by $(i+j)\bmod n$, i.e.\ by the remainder (that belongs to the set $[n]$) left when $(i+j)$ is divided by $n$.
Henceforth, we let $a^{\ell}$, for any $a \in \mathcal{A}$ and $\ell \in \mathbb{N}$, indicate a subsequence of length $\ell$, each of whose terms equals the symbol $a$. If $\ell=1$, we simply write $a$, instead of $a^{1}$. For $a_{1}, a_{2}, \ldots, a_{r} \in \mathcal{A}$, and $\ell_{1}, \ell_{2}, \ldots, \ell_{r} \in \mathbb{N}$, we let $a_{1}^{\ell_{1}}a_{2}^{\ell_{2}} \ldots a_{r}^{\ell_{r}}$ indicate a subsequence in which each of the first $\ell_{1}$ terms equals $a_{1}$, each of the next $\ell_{2}$ terms equals $a_{2}$, and so on. 
For instance, the notation $01^{2}0^{2}1$ indicates the subsequence $(0,1,1,0,0,1)$. 

The \pca{}, along with an \emph{initial configuration} $\eta(0)$ (which may or may not be randomly chosen from the state space $\Omega$), generates a discrete-time-indexed stochastic process $\{\eta(t)\}_{t \in \mathbb{N}_{0}}$, with $\eta(t)=(\eta_{i}(t): i \in [n])$, where $\eta_{i}(t)$ indicates the symbol, from $\mathcal{A}$, that occupies the cell $i$ during epoch $t$. We refer to $\eta_{i}(t)$ as the \emph{state} of the cell $i$ during epoch $t$. Conditioned on $\eta(t)$, the state of $i$ is updated to $\eta_{i}(t+1)$ during epoch $(t+1)$, independent of all other cells in $[n]$, as follows (the `R' in the itemization below stands for `rule', as these provide the stochastic update rules for the \pca{}):
\begin{enumerate}[label=(R\arabic*), ref=R\arabic*]
\item \label{rule_1} If $\big(\eta_{i}(t),\eta_{i+1}(t),\ldots,\eta_{i+m-1}(t)\big)=0^{m}$, we set
\[
 \eta_{i}(t+1) = 
  \begin{cases} 
   0 & \text{with probability } (1-p_1), \\
   1 & \text{with probability } p_1.
  \end{cases}
\]
\item \label{rule_2} If $\big(\eta_{i}(t),\eta_{i+1}(t),\ldots,\eta_{i+m-1}(t)\big)=0^{m-1}1$, we set
\[
 \eta_{i}(t+1) = 
  \begin{cases} 
   0 & \text{with probability } p_2, \\
   1 & \text{with probability } (1-p_2).
  \end{cases}
\]
\item \label{rule_3} In all other cases, i.e.\ when $\big(\eta_{i}(t),\eta_{i+1}(t),\ldots,\eta_{i+m-1}(t)\big)\in \mathcal{A}^{m}\setminus \{0^{m},0^{m-1}1\}$, we set $\eta_{i}(t+1)=0$ with probability $1$.
\end{enumerate}

If we set $m = 2$, $p_2 = 1$ in our model and change the state space to 
$\mathbb{Z}$, we get the well-known \emph{directed animals PCA}~\cite [Figure 7]{mairesse-marcovici-2014}.

In particular, note that, for any cell $j \in [n]$, the probability of the event that $\eta_{j}(t+1)=1$, conditioned on $\eta(t)$, is strictly positive only if $\eta_{j}(t)=\eta_{j+1}(t)=\cdots=\eta_{j+m-2}(t)=0$. In other words:
\begin{equation}\label{necessary_cond_for_1}
\Prob\left[\eta_{j}(t+1)=1\big|\eta(t)\right]>0 \implies \eta_{i}(t)=0 \text{ for each } i \in \{j,j+1,\ldots,j+m-2\}.
\end{equation}
In fact, if $p_2=1$ (which, as we shall see, is permitted since in our main results, \cref{thm:main_1} and \cref{thm:main_2}, we have assumed $p_2 \in (0,1]$), we can further assert that for the event $\eta_{j}(t+1)=1$ to take place with positive probability, for any cell $j \in [n]$, we must have $\eta_{i}(t)=0$ for each $i \in \{j,j+1,\ldots,j+m-1\}$. However, in our analysis in this paper, we do not need to consider this special case separately.

A pictorial representation of these stochastic update rules is captured in \cref{fig:rules}.

\begin{figure}[h!]
  \centering
\fbox{    \includegraphics[width=0.8\textwidth]{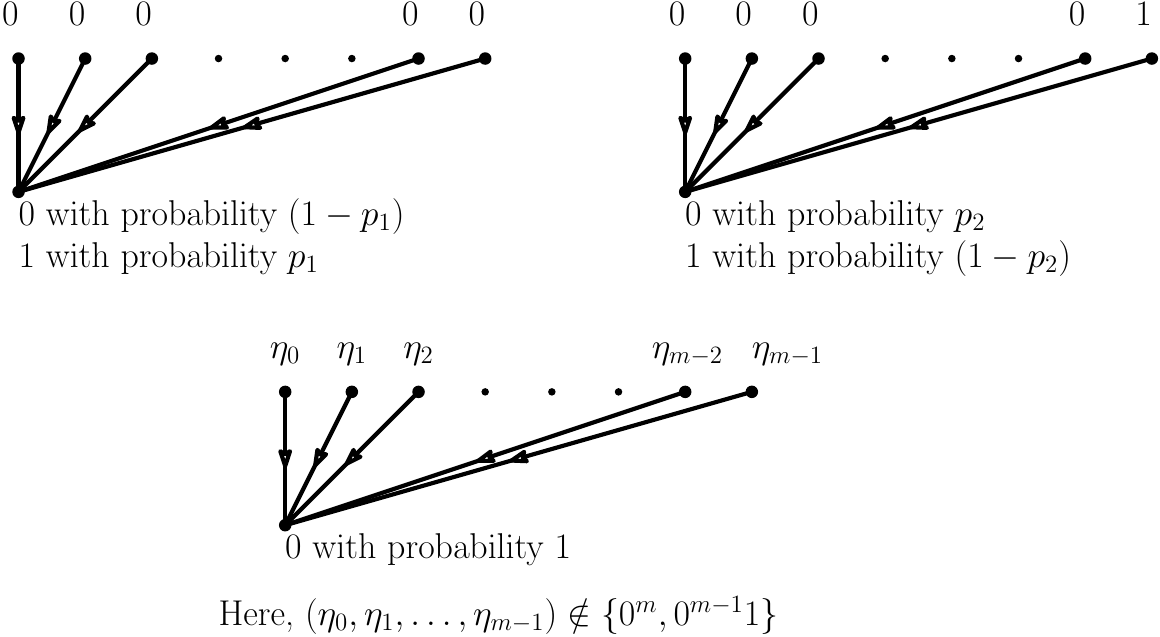}}
\caption{A pictorial representation of the stochastic update rules of the \pca{}.}
  \label{fig:rules}
\end{figure}

The transition graphs for the Markov chains resulting from the $m$-NED model with $n = 3$ and both $m = 2$ and $m = 3$ have been shown in \cref{fig:transition graph eg}.

\begin{figure}[h!]
\centering
\includegraphics[scale=0.5]{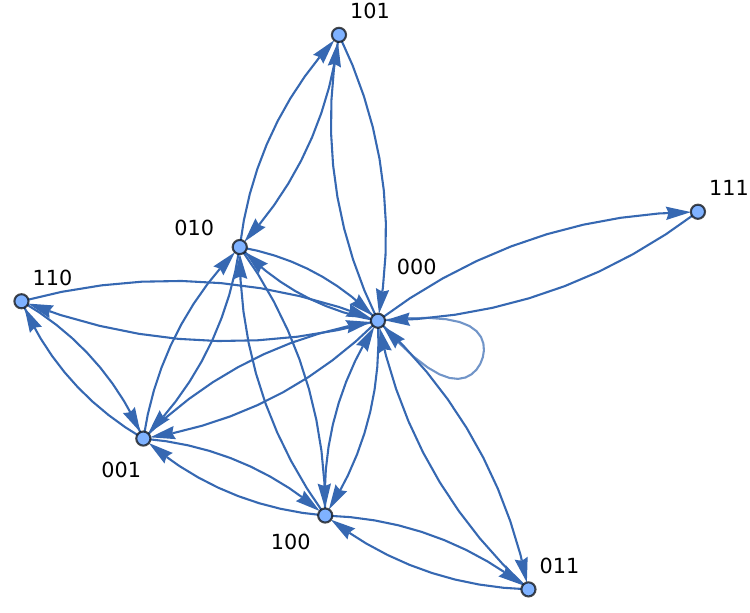}
\includegraphics[scale=0.5]{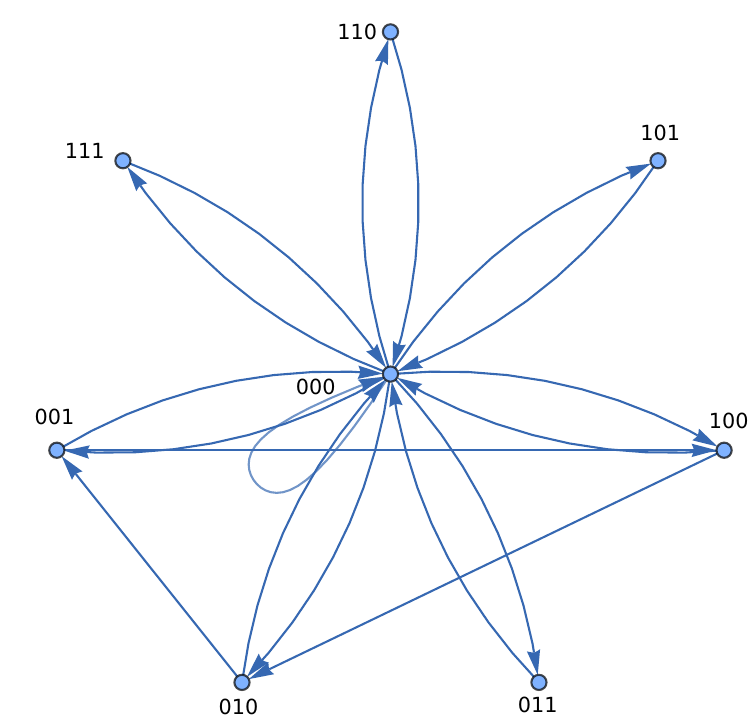}
\caption{All transitions for the \pca{} with $n = 3$ and $m = 2$ (resp. $m = 3$) on the left (resp. right). The probabilities of the transitions have not been written down to avoid cluttering the figure.}
\label{fig:transition graph eg}
\end{figure}

If $\eta(0)$ is random, its law, say $\nu_{0}$, is referred to as the \emph{initial distribution} (if $\eta(0)$ is deterministic, $\nu_{0}$ is simply the Dirac probability measure $\delta_{\eta(0)}$). Evidently, the stochastic process $\{\eta(t)\}_{t \in \mathbb{N}_{0}}$ is a Markov chain on a finite state space, and the objective of this paper is to explore the \emph{limiting distribution} (which, we show, exists and is unique) of this Markov chain. In other words, letting $\nu_{t}$ denote the law of $\eta(t)$, we establish the existence of, and provide a complete description of, the unique probability distribution $\pi$ such that $\lim_{t \rightarrow \infty}\nu_{t}=\pi$ irrespective of what $\nu_{0}$ is.

\section{Main results}
\label{sec:main_results}

As mentioned above, our main results concern themselves with the existence and complete characterization of the limiting or stationary distribution $\pi$ for the 
\pca{}.
Henceforth, for configurations $\alpha, \beta \in \Omega$, we denote by $\Prob[\alpha \rightarrow \beta]$ the transition probability $\Prob[\eta(t+1)=\beta\big|\eta(t)=\alpha]$, for \emph{any} $t \in \mathbb{N}_{0}$ (since $\{\eta(t)\}_{t \in \mathbb{N}_{0}}$ is a time-homogeneous Markov chain). 
Recall that the limiting distribution satisfies the \emph{balance equation} or 
\emph{master equation},
\begin{equation}
\sum_{\alpha\in\Omega}\Prob[\alpha \rightarrow \beta] \pi(\alpha)=\pi(\beta) \quad\text{for each } \beta \in \Omega.\label{gen_identity}
\end{equation}

Before we describe the stationary distribution, we state an elementary property of the \pca{}.

\begin{prop}
For any $n \geqslant m \geqslant 2$ and $p_1, p_2 \in (0, 1)$, the \pca{} is translation-invariant.
\end{prop}

\begin{proof}
This is an immediate consequence of the fact that the transition probabilities in \eqref{rule_1}, \eqref{rule_2} and \eqref{rule_3} do not depend on the sites.
\end{proof}

As a consequence, we obtain the following.

\begin{corollary}
\label{cor:trans inv}
The stationary probability of $\beta = (\beta_1, \dots, \beta_n)$ satisfies the identity, $\pi(\beta) = \pi(\beta_2, \beta_3, \allowbreak \dots, \beta_n, \beta_1)$.
\end{corollary}

Given any configuration $\beta=(\beta_{1},\beta_{2},\ldots,\beta_{n}) \in \Omega$, we introduce the following notation (the `N' in the itemization below stands for `notation'):
\begin{enumerate}[label=(N\arabic*), ref=N\arabic*]
\item \label{notation:1} $N_{1}(\beta)$ denotes the number of occurrences of the symbol $1$ in $\beta$, i.e.\ 
\begin{equation}
N_{1}(\beta)=\sum_{i \in [n]} \chi\left[\beta_{i}=1\right];\nonumber
\end{equation}
where the notation $\chi[A]$, given any event $A$, denotes the indicator for the event $A$;
\item \label{notation:10^{r}1} $N_{10^{r}1}(\beta)$, for any $r\in\mathbb{N}$, denotes the number of occurrences of the subsequence $10^{r}1$ in $\beta$, i.e.\
\begin{equation}
N_{10^{r}1}(\beta)=\sum_{i \in [n]} \chi\left[(\beta_{i},\beta_{i+1},\ldots, \beta_{i+r+1})=10^{r}1\right];\nonumber
\end{equation}
\item \label{notation:0^{r}1} $N_{0^{r}1}(\beta)$, for any $r\in\mathbb{N}$, denotes the number of occurrences of the subsequence $0^{r}1$ in $\beta$, i.e.\
\begin{equation}
N_{0^{r}1}(\beta)=\sum_{i \in [n]} \chi\left[(\beta_{i},\beta_{i+1},\ldots,\beta_{i+r})=0^{r}1\right].\nonumber
\end{equation}
\end{enumerate}

\begin{theorem}\label{thm:main_1}
As long as $p_1 \in (0,1)$ and $p_2 > 0$, there exists a unique probability distribution $\pi$, supported on the state space $\Omega$, such that the sequence $\{\nu_{t}\}_{t \in \mathbb{N}_{0}}$ of probability distributions converges to $\pi$ as $t \rightarrow \infty$, irrespective of what the initial distribution $\nu_{0}$ is. For each configuration $\beta \in \Omega$, we have
\begin{equation}
\pi(\beta)=\frac{1}{Z_{n,m}}p_1^{N_{1}(\beta)}(1-p_1)^{\sum_{r\in[m-2]}r N_{10^{r}1}(\beta)+(m-1)N_{0^{m-1}1}(\beta)}p_2^{-N_{0^{m-1}1}(\beta)},\label{limiting_distribution}
\end{equation}
where $Z_{n,m} \equiv Z_{n,m}(p_1,p_2)$ is the normalizing constant. 
\end{theorem}

The proof of \cref{thm:main_1} forms the bulk of the paper and is proved in \cref{sec:proof_main_1}.

\begin{example}
\label{eg:m2n3}
For the example of $m = 2$ and $n = 3$ shown in \cref{fig:transition graph eg},
the stationary distribution written as a vector, with the configurations ordered lexicographically, is
\[
\frac{1}{Z_{3,2}} 
\left(1, \frac{(1 - p_1) p_1}{p_2}, \frac{(1 - p_1) p_1}{p_2}, \frac{(1 - p_1) p_1^2}{p_2}, 
\frac{(1 - p_1) p_1}{p_2}, \frac{(1 - p_1) p_1^2}{p_2}, \frac{(1 - p_1) p_1^2}{p_2}, p_1^3 \right),
\]
where
\[
Z_{3,2} = \frac{1}{p_2} (1 + p_1) (3 p_1 - 3 p_1^2 + p_2 - p_1 p_2 + p_1^2 p_2).
\]
\end{example}

Recall that a Markov chain is said to be \textit{reversible} if any stationary distribution satisfies the \emph{detailed balance} equation,
\begin{equation}
\label{detailed balance}
\pi(\alpha) \Prob[\alpha \to \beta] = \pi(\beta) \Prob[\beta \to \alpha]
\quad
\text{for all $\alpha, \beta \in \Omega$}.
\end{equation}

\begin{corollary}
\label{cor:detailed m>2}
The \pca{}
is irreversible for any $n \geqslant m \geqslant 3$ and $p_1, p_2 \in (0, 1)$. 
\end{corollary}

\begin{proof}
For $n \geqslant m \geqslant 3$, there exist configurations $\alpha, \beta \in \Omega$ such that there is a transition from $\alpha$ to $\beta$, but not the other way around. For example, take $\alpha = 0^{n-2}10$ and $\beta = 0^{n-1}1$.
Therefore, detailed balance cannot hold. 
\end{proof}

The normalizing constant $Z_{n,m} \equiv Z_{n,m}(p_1,p_2)$ is commonly referred to as the \emph{partition function}, and \cref{thm:main_2} is concerned with the description of $Z_{n,m}$. 
Define 
\begin{equation}
\mathcal{T}_{M}=\left\{(x_{1},x_{2},\ldots,x_{m-2})\in\{0,1,\ldots,k\}^{m-2} 
\mid \sum_{s\in[m-2]}s x_{s}=M\right\}, \label{T_{M}}
\end{equation}
and define $\mathcal{C}_{n,k}\subseteq\mathbb{N}_{0}^{2}$ to be the set of all pairs $(M,N)$ such that
\begin{equation}
M\in\{0,1,\ldots,n-k-m+1,n-k\},\quad 0\leqslant N \leqslant \left\lfloor\frac{n-k-M}{m-1}\right\rfloor, \label{M_N_choices}
\end{equation}
with $N=0$ if and only if $M=(n-k)$.
Note that the set $\{0,1,\ldots,n-k-m+1,n-k\}$ simply boils down to the singleton $\{n-k\}$ if $(n-k)\leqslant(m-1)$. 
In what follows, we use the standard convention $\binom{-1 }{ -1}=1$. 

\begin{theorem}\label{thm:main_2}
The partition function $Z_{n,m}$, mentioned in \eqref{limiting_distribution}, is given by
\begin{multline}
Z_{n,m}=1+\sum_{k=1}^{n} \,\sum_{(M,N)\in \mathcal{C}_{n,k}} \;
\sum_{(x_1,\dots,x_{m-2}) \in\mathcal{T}_{M}} \;
\frac{n}{k} 
\binom{k}{x_1,\dots,x_{m-2},N, k - N - x_1 - \cdots - x_{m-2}} \\ 
\times \binom{n-k-M-(m-2)N-1}{ N-1}p_1^{k}(1-p_1)^{M+(m-1)N}p_2^{-N}.
\label{Z_{n,m}_final}
\end{multline}
\end{theorem}

To compare with \cref{eg:m2n3}, we compute $Z_{3,2}$ using \cref{thm:main_2}. 
Note that $m = 2$ forces $M = 0$ for $(M, N) \in \mathcal{C}_{3,k}$ 
for all $k \in [3]$ because of $\mathcal{T}_M$. Thus, the only terms that contribute are for $k =1, N = 1$, $k = 2, N = 1$ and $k = 3,  = 0$.
This gives
\[
Z_{3, 2} = 1 + 3 \frac{p_1 (1 - p_1)}{p_2} + 3 \frac{p_1 (1 - p_1)}{p_2} + p_1^3,
\]
which matches \cref{eg:m2n3} after factoring.

The next result gives a formula for the \textit{density} in the stationarity distribution $\pi$, namely the probability that a particular site, say $1$, is occupied by a particle.

\begin{theorem}\label{thm:limiting_prob_cell_occupied_by_1}
Let $\eta \in \Omega$ be a random configuration with law $\pi$.
Then 
\begin{multline}
\pi(\eta_{1}=1) = \frac{1}{Z_{n,m}}
\sum_{k=1}^{n} \,\sum_{(M,N)\in \mathcal{C}_{n,k}} \;
\sum_{(x_1,\dots,x_{m-2}) \in\mathcal{T}_{M}} \;
\binom{k}{x_1,\dots,x_{m-2}, N, k - N - x_1 - \cdots - x_{m-2}} \\
\times \binom{n-k-M-(m-2)N-1}{N-1}p_1^{k}(1-p_1)^{M+(m-1)N}p_2^{-N}.\label{eq:limiting_prob_cell_occupied_by_1}
\end{multline}
\end{theorem}

We prove \cref{thm:main_2,thm:limiting_prob_cell_occupied_by_1} in \cref{sec:proof_main_2}.

We next focus on the neighbourhood size $m=2$, where we are able to prove more explicit formulas.

\begin{prop}
\label{prop:detailed m=2}
The \pca{} with $m=2$
is reversible if and only if $n = 2$, or 
$n > 2$ with $p_1 + p_2 = 1$. 
\end{prop}

We now obtain a more compact expression for the partition function $Z_{n,2}
\equiv Z_{n,2}(p_1,p_2)$, which is captured by the following theorem:
\begin{theorem}
\label{thm:main_3}
We fix $p_1\in(0,1)$ and $p_2>0$. When $m=2$, the sequence $\{Z_{n,2}\}$ of partition functions of the \pca{} satisfies the recurrence relation
\begin{equation}\label{partition_recurrence_m=2}
p_1(1-p_1-p_2)Z_{n,2}+p_2(1+p_1)Z_{n+1,2}=p_2 Z_{n+2,2} \quad \text{for each } n\geqslant 2,
\end{equation}
with $Z_{0,2} = 2$ and $Z_{1,2} = 1 + p_1$. Equivalently, the generating function for the  sequence of partition functions, $\{Z_{n,2}\}$, is given by
\begin{equation}
\label{partition_function_gf}
\sum_{n=0}^{\infty}Z_{n,2}x^{n}=\frac{(2-x-p_1x)p_2}{p_2-p_2x(1+p_1)-x^{2}p_1(1-p_1-p_2)}.
\end{equation}
\end{theorem}

The \emph{free energy} is defined to be
\[
F(m, p_1, p_2) = \lim_{n \to \infty} \frac{1}{n} \log Z_{n,m}(p_1, p_2).
\]

\begin{theorem}
\label{thm:free energy m=2}
The free energy of the \pca{} when $m = 2$ can be written as
\[
F(2, p_1, p_2) = -\log \left( 
\frac{-p_2(1+p_1) + \sqrt{p_2 (1-p_1)(4 p_1 + p_2 - p_1 p_2) }}
{2 p_1(1 - p_1 - p_2)} \right),
\]
where the line $p_1 + p_2 = 1$ is a removable singularity. 
In the latter case, we get
\[
F(2, p_1, 1-p_1) = \log(1 + p_1).
\]
\end{theorem}

See \cref{fig:free energy} for a contour plot of the free energy.

\begin{center}
\begin{figure}[h!]
\includegraphics[scale=0.6]{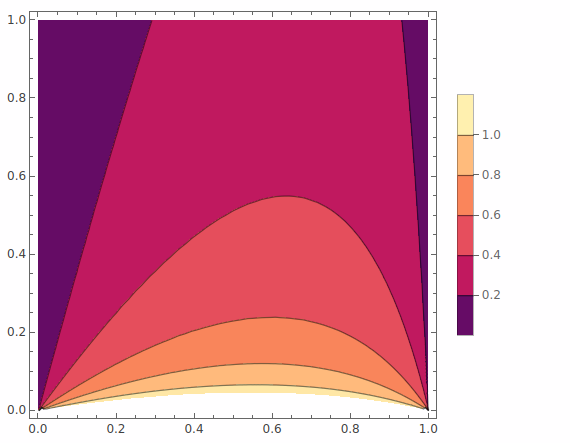}
\caption{A contour plot of the free energy.}
\label{fig:free energy}
\end{figure}
\end{center}

\begin{theorem}
\label{thm:density m=2}
When $m=2$, the generating function for the sequence 
$\{Z_{n,2} \pi(\eta_{1}=1)\}_{n}$ is given by 
\begin{equation}
\sum_{n=0}^{\infty}Z_{n,2} \pi(\eta_{1}=1) x^{n} =
\frac{p_{1}x(1-x+q_{2}x)}{(1-x)(1-p_{1}x)-p_{1}q_{2}x^{2}},
\label{eq:gen_fn_limiting_prob_cell_occupied_by_1}
\end{equation}
where $q_{2}=p_{2}^{-1}(1-p_{1})$.
\end{theorem}

The proofs of \cref{prop:detailed m=2}, and \cref{thm:main_3,thm:free energy m=2,thm:density m=2} are given in \cref{sec:m=2}.

\section{Stationary distribution}
\label{sec:proof_main_1}

We begin by establishing the existence of a unique stationary distribution for the 
\pca{}
when $p_1 \in (0,1)$ and $p_2>0$. Given any configuration $\beta\in\Omega$ and any two integers $j_{1}\in[n]\cup\{0\}$ and $j_{2}\in[n]$ such that $j_{1}<j_{2}$, we set
\begin{equation}
\beta_{[j_{1}+1,j_{2}]}=\left(\beta_{j_{1}+1},\beta_{j_{1}+2},\ldots,\beta_{j_{2}}\right).\label{partial_configuration}
\end{equation}
We abbreviate this notation to simply $\beta_{[j_{2}]}$ if $j_{1}=0$. 
Recall that for configurations $\alpha, \beta \in \Omega$, we denote by $\Prob[\alpha \rightarrow \beta]$ the transition probability $\Prob[\eta(t+1)=\beta\big|\eta(t)=\alpha]$. 
Even more generally, given indices $j_{1}\in[n]\cup\{0\}$ and $j_{2}\in[n]$ with $j_{1} < j_{2}$, we define
\begin{equation}
\Prob[\alpha\rightarrow\beta_{[j_{1}+1,j_{2}]}] = \prod_{j\in\{j_{1}+1,j_{1}+2,\ldots,j_{2}\}}\Prob\left[\eta_{j}(t+1)=\beta_{j}\big|\eta(t)=\alpha\right].\label{partial_transition_probability}
\end{equation}
In particular, if $j_{1}=0$, we abbreviate the notation in \eqref{partial_transition_probability} to 
$\Prob[\alpha\rightarrow\beta_{[j_{2}]}]$, and we write $\Prob[\alpha\rightarrow\beta_{[j_2, j_{2}]}] = 
\Prob[\alpha\rightarrow\beta_{j_{2}}]$.

\begin{lemma}
\label{lem:irreducible_aperiodic}
The Markov chain $\{\eta(t)\}_{t \in \mathbb{N}_{0}}$ is irreducible and aperiodic as long as $p_1$ lies in the open interval $(0,1)$ and $p_2$ is strictly positive.
\end{lemma}

\begin{proof}
The proof relies on the observation that any configuration $\beta=(\beta_{i}: i \in [n])$, belonging to the state space $\Omega$, can be `reached' from the configuration $0^{n}$ (in which each cell of $[n]$ is occupied by the symbol $0$) in a single epoch, and conversely, any configuration $\beta \in \Omega$ can `lead to' the configuration $0^{n}$ in a single epoch. Formally speaking, the transition probability $\Prob[0^{n}\rightarrow \beta]$ is strictly positive because of \eqref{rule_1} and the assumption that $p_1\in(0,1)$, whereas the transition probability $\Prob[\beta \rightarrow 0^{n}]$ is strictly positive because of all three of \eqref{rule_1}, \eqref{rule_2} and \eqref{rule_3}, and the assumptions that $p_1<1$ and $p_2>0$. While this is enough to establish irreducibility, the claim of aperiodicity is established by noting that the transition probability $\Prob[0^{n}\rightarrow 0^{n}]$ is also strictly positive.
\end{proof}

\cref{lem:irreducible_aperiodic} is enough to establish the existence of a \emph{unique} stationary probability distribution $\pi$ on $\Omega$ such that $\lim_{t \rightarrow \infty}\nu_{t}=\pi$ for all initial distributions $\nu_{0}$ (see, for instance, results from \cite[Chapter 3]{karlin-2014}). Moreover, $\pi$ is the unique probability distribution on $\Omega$ that satisfies \eqref{gen_identity}.
Consequently, to establish \cref{thm:main_1}, it suffices for us to show that the expression for $\pi$ given by \eqref{limiting_distribution} satisfies the master equation in \eqref{gen_identity}. This is what we accomplish in the rest of \cref{sec:proof_main_1}. 

Fix an arbitrary configuration $\beta=(\beta_{1},\beta_{2},\ldots,\beta_{n}) \in \Omega$, and let us define $\mathcal{S}_{\beta}$ to be the subset of $\Omega$ consisting of all those configurations $\alpha$ such that $\Prob[\alpha \rightarrow \beta]>0$. 
Let $N_{0^{r}}(\beta)$, for any $r\in\mathbb{N}$, denote the number of occurrences of the subsequence $0^{r}$ in $\beta$, i.e.\
\begin{equation}
N_{0^{r}}(\beta)=\sum_{i \in [n]} \chi\left[(\beta_{i},\beta_{i+1},\ldots,\beta_{i+r-1})=0^{r}\right].\nonumber
\end{equation}

We fix $p_1 \in (0,1)$ and $p_2 \in (0,1]$. Using \cref{cor:trans inv}, we can write $\beta$ uniquely as 
\begin{equation}\label{beta_decomposed}
\beta=\underbrace{1^{k_{1}}0^{\ell_{1}}}\underbrace{1^{k_{2}}0^{\ell_{2}}}\ldots \underbrace{1^{k_{R-1}}0^{\ell_{R-1}}}\underbrace{1^{k_{R}}0^{\ell_{R}}},
\end{equation}
for some $R \in \mathbb{N}$ and some $k_{1}, \ell_{1}, \ldots, k_{R}, \ell_{R} \in \mathbb{N}$ as long as $\beta \in \Omega\setminus\{0^{n},1^{n}\}$.
When $\beta=1^{n}$, we set $R=1$, $k_{1}=n$ and $\ell_{1}=0$, whereas when $\beta=0^{n}$, we set $R=1$, $k_{1}=0$ and $\ell_{1}=n$. With $\beta$ written in the form stated in \eqref{beta_decomposed}, and with $s_{i}$ defined as
\begin{equation}
s_{0}=0 \quad \text{and} \quad s_{i}= \sum_{j \in [i]} \left(k_{j}+\ell_{j}\right) \text{ for each } i \in [R],\label{s_{i}_defns}
\end{equation}
we refer to the tuple $(s_{i-1}+1,s_{i-1}+2,\ldots,s_{i})$ of cells as the $i$-th \emph{block} of the configuration $\beta$. Note that the $i$-th block of $\beta$ consists of $k_{i}$ consecutive occurrences of the symbol $1$, followed by $\ell_{i}$ consecutive occurrences of the symbol $0$, i.e.\
\begin{equation}
\beta_{s_{i-1}+1}=\cdots=\beta_{s_{i-1}+k_{i}}=1 \quad \text{and} \quad \beta_{s_{i-1}+k_{i}+1}=\cdots=\beta_{s_{i}}=0.\nonumber
\end{equation}
In particular, when $\beta\in\{0^{n},1^{n}\}$, the entire configuration $\beta$ forms a single block. 

As mentioned right before stating the stochastic update rules in \eqref{rule_1}, \eqref{rule_2} and \eqref{rule_3}, conditioned on the configuration $\eta(t)$, the state of the cell $j$ is updated to $\eta_{j}(t+1)$ independent of all else, for each $j\in[n]$, so that we can write, using the notation introduced in \eqref{partial_transition_probability}, for each $\alpha \in \mathcal{S}_{\beta}$, 
\begin{equation}
\Prob[\alpha \rightarrow \beta]=\prod_{j \in [n]}\Prob\left[\eta_{j}(t+1)=\beta_{j}\big|\eta(t)=\alpha\right]=\prod_{i\in[R]}\Prob[\alpha\rightarrow\beta_{[s_{i-1}+1,s_{i}]}].\label{transition_probability_decomposed}
\end{equation}
It is evident, from \eqref{transition_probability_decomposed}, that it suffices for us to focus on $\Prob[\alpha\rightarrow\beta_{[s_{1}]}]$, which corresponds to the first block of $\beta$, as the analysis can then be generalized to any of the remaining $(R-1)$ blocks of $\beta$ using \cref{cor:trans inv}. 
For the sake of brevity of notation, we replace $k_{1}$ by $k$ and $\ell_{1}$ by $\ell$ in some of the analysis that follows, though eventually, we revert back to $k_{1}$ and $\ell_{1}$ when putting together our findings to obtain a final expression for \eqref{transition_probability_decomposed}. With this updated notation, the first block of $\beta$ corresponds to the tuple $\beta_{[k+\ell]}$ consisting of
\begin{equation}\label{beta_first_block}
\beta_{i}=1 \text{ for each } i \in [k] \quad \text{and} \quad \beta_{i}=0 \text{ for each } i \in [k+1,k+\ell].
\end{equation}
Note that if $\beta$ comprises a single block, we have $k+\ell=n$, and the entire structure of $\beta$ is captured by \eqref{beta_first_block}, whereas if $\beta$ comprises at least two blocks, we have $k+\ell+1\leqslant n$, and from \eqref{beta_decomposed}, the first element of the second block of $\beta$ is given by $\beta_{k+\ell+1}=1$. 

We will first prove the following lemma as a key part of the proof of \cref{thm:main_1}.
\begin{lemma}
\label{lem:gen identity}
For any configuration $\beta \in \Omega \setminus \{0^n\}$,
\begin{align}
\sum_{\alpha \in\mathcal{S}_{\beta}}\Prob[\alpha \rightarrow \beta] \pi(\alpha)=
&\frac{(1-p_1)^{n}}{Z_{n,m}}
\sum_{\alpha \in\mathcal{S}_{\beta}} 
\left( \frac{p_1}{1-p_1} \right)^{N_{1}(\alpha) + N_{1}(\beta)}
\nonumber\\
& \times
\left(\frac{(1-p_1)(1-p_2)}{p_1p_2}\right)^{\left|\left\{i \in [R]:\ell_{i}\geqslant m-1,\alpha_{s_{i-1}+k_{i}+m-1}=1\right\}\right|}.
\label{gen_identity_simplified_step}
\end{align}
\end{lemma}

The proof of \cref{lem:gen identity} can be divided into several broad steps, and these have been established in detail in the sequel. However, for the convenience of the reader, we outline here the heuristics of the proof, which involves fixing any configuration $\beta\in\Omega$ and considering all the transitions leading to $\beta$. 
\begin{enumerate}
\item We begin by considering $\beta=0^{n}$, which has been dealt with in \cref{subsec:beta_all_0}. This involves computing $\Prob\left[\alpha\rightarrow 0^{n}\right]$ for each $\alpha\in\Omega$, which has been accomplished in \cref{lem:transition_probabilities_beta_all_0}, followed by expressing it in terms of $N_{1}(\alpha)$, $N_{10^{r}1}(\alpha)$ for $r\in[m-2]$ and $N_{0^{m-1}1}(\alpha)$, as shown in \eqref{transition_eq_beta_all_0}, and finally verifying \eqref{gen_identity} by substituting \eqref{transition_eq_beta_all_0} and \eqref{limiting_distribution}.
\item Our analysis for the case of $\beta\neq 0^{n}$ is accomplished in \cref{subsec:beta_not_all_0}. We compute $\Prob[\alpha\rightarrow\beta_{[k]}]$ in \cref{lem:transition_probabilities_first_k_coordinates}, and we compute $\Prob[\alpha\rightarrow\beta_{[k+\ell]}]$, when $\ell\leqslant(m-2)$, in \cref{subsec:ell_less_equal_m-2}. When $\ell\geqslant(m-1)$, we introduce a decomposition of $\alpha_{[k+1,k+\ell]}$, as shown in \eqref{alpha'_decomposed} of \cref{subsec:ell_greater_m-2_beta_not_0}, then deduce the transition probabilities corresponding to the various parts of this decomposition in \cref{lem:transition_probabilities_alpha_{[k+1,k+ell]}}. An expression for $\Prob[\alpha\rightarrow\beta_{[k+\ell]}]$, when $\beta\neq 0^{n}$ and $\ell\geqslant(m-1)$, is then provided in \cref{cor:0_length_first_block_geq_m-1}, and in \cref{lem:identity} of \cref{subsec:express}, this expression has been expressed in terms of $N_{1}, N_{10^{r}1}$ and $N_{0^{m-1}1}$.
\item The final step of the proof of \cref{lem:gen identity} has been established in \cref{subsec:verification_master_eq} using the expression stated in \cref{lem:identity}.
\end{enumerate}

We shall complete the proof of \cref{thm:main_1} in \cref{subsec:complete}.

\subsection{The case of $\beta=0^{n}$}\label{subsec:beta_all_0} 
We begin with the relatively simpler case of $\beta=0^{n}$, wherein $\alpha_{i}$ is allowed to be either $0$ or $1$ for each cell $i \in [n]$ because of \eqref{rule_1}, \eqref{rule_2} and \eqref{rule_3}. We may, then, write $\alpha$ as
\begin{equation}
\alpha=0^{g_{1}}1^{h_{1}}0^{g_{2}}1^{h_{2}}\ldots 0^{g_{S}}1^{h_{S}},\label{alpha_decomposed_beta_all_0}
\end{equation}
where $S \in \mathbb{N}$ and each of $g_{1}, h_{1}, \ldots, g_{S}, h_{S} \in \mathbb{N}_{0}$, such that, for $\alpha \in \Omega\setminus\{0^{n},1^{n}\}$, whenever $g_{i}>0$, we have $h_{i}>0$, for each $i \in [S]$. In particular, we set $S=1$, $g_{1}=n$ and $h_{1}=0$ if $\alpha=0^{n}$, and we set $S=1$, $g_{1}=0$ and $h_{1}=n$ if $\alpha=1^{n}$. We also observe here that $g_{1}=0$ if and only if $\alpha=1^{n}$ and $h_{S}=0$ if and only if $\alpha=0^{n}$ -- in other words, as long as $\alpha\neq 1^{n}$ and $\alpha\neq0^{n}$, the representation of $\alpha$ in \eqref{alpha_decomposed_beta_all_0} must begin with the symbol $0$ and end in the symbol $1$. We define
\begin{equation}
\sigma_{0}=0 \quad \text{and} \quad \sigma_{i}=\sum_{j\in[i]}\left(g_{j}+h_{j}\right) \text{ for each } i \in [S].\label{sigma_i_defns_beta_all_0}
\end{equation}
Naturally, we refer to the tuple $\left(\sigma_{i-1}+1,\sigma_{i-1}+2,\ldots,\sigma_{i-1}+g_{i}\right)$ as the $i$-th \emph{$0$-string} of $\alpha$. The \emph{length} of this string equals $g_{i}$. Likewise, if $\alpha\neq 0^{\ell}$, we have $h_{i}\geqslant 1$ for each $i \in [S]$, and $\alpha_{\sigma_{i-1}+g_{i}+1}=\cdots=\alpha_{\sigma_{i}}=1$ for each $i \in [S]$. We, therefore, refer to $\left(\sigma_{i-1}+g_{i}+1,\ldots,\sigma_{i}\right)$ as the $i$-th \emph{$1$-string} of $\alpha$, for each $i \in [S]$, and its \emph{length} equals $h_{i}$.

\begin{lemma}\label{lem:transition_probabilities_beta_all_0}
When $\beta=0^{n}$ and $\alpha$ is represented as in \eqref{alpha_decomposed_beta_all_0}, with $\alpha\neq 0^{n}$, the transition probabilities corresponding to the sites belonging to the $1$-strings of $\alpha$ are given by
\begin{equation}
\prod_{j\in\left[\sigma_{i-1}+g_{i}+1,\sigma_{i}\right]}
\Prob[\alpha \to \beta_{j}=0] =1 \text{ for each } i\in[S], \label{transition_eq_1_strings_beta_all_0}
\end{equation}
and the transition probabilities corresponding to the sites belonging to the $0$-strings of $\alpha$ equal
\begin{equation}\label{transition_eq_0_strings_beta_all_0}
\prod_{j\in\left[\sigma_{i-1}+1,\sigma_{i-1}+g_{i}\right]}
\Prob[\alpha \to \beta_{j}=0] =
\begin{cases}
1 &\text{for each }i\in[S] \text{ with }g_{i} \leqslant (m-2),\\
(1-p_1)^{g_{i}-m+1}p_2 &\text{for each }i\in[S] \text{ with }g_{i} \geqslant (m-1).
\end{cases}
\end{equation}
When $\alpha=0^{n}$, we have $\Prob\left[\alpha\rightarrow\beta\right]=\Prob\left[0^n \rightarrow 0^n\right]=(1-p_1)^{n}$.
\end{lemma}
\begin{proof}
The identity in \eqref{transition_eq_1_strings_beta_all_0} is an immediate consequence of \eqref{necessary_cond_for_1}, and note that the case of $\alpha=1^{n}$ is taken care of by \eqref{transition_eq_1_strings_beta_all_0}. Next, we assume that $\alpha\notin\left\{0^{n},1^{n}\right\}$. Recall that $\alpha_{\sigma_{i-1}+g_{i}+1}=1$ for each $i \in [S]$, which follows from the observation (made after \eqref{sigma_i_defns_beta_all_0}) that $h_{i}\geqslant 1$ for each $i\in[S]$. This, along with \eqref{necessary_cond_for_1}, ensures that whenever $g_{i} \leqslant (m-2)$, we have  
\begin{equation}
\Prob[\alpha \to \beta_{j}=0] =1\text{ for each } j \in \left\{\sigma_{i-1}+1,\sigma_{i-1}+2,\ldots,\sigma_{i-1}+g_{i}\right\},\nonumber
\label{transition_eq_10}
\end{equation}
thus yielding the first identity of \eqref{transition_eq_0_strings_beta_all_0}. Yet another application of \eqref{necessary_cond_for_1} allows us to conclude that, for $i \in [S]$ with $g_{i}\geqslant(m-1)$, we have
\begin{equation}
\Prob[\alpha \to \beta_j = 0]
=1 \text{ for each } j\in \left\{\sigma_{i-1}+g_{i}-m+3,\ldots,\sigma_{i-1}+g_{i}\right\},\label{transition_eq_3_beta_0}
\end{equation}
while by \eqref{rule_1}, we obtain 
\begin{equation}
\Prob[\alpha \to \beta_j = 0]
=(1-p_1) \text{ for each } j\in\left\{\sigma_{i-1}+1,\ldots,\sigma_{i-1}+g_{i}-m+1\right\}.\label{transition_eq_4_beta_0}
\end{equation}
Note that if $g_{i}=(m-1)$, the set $\left\{\sigma_{i-1}+1,\ldots,\sigma_{i-1}+g_{i}-m+1\right\}$ is empty, and the transition probabilities in \eqref{transition_eq_4_beta_0} would not matter to us. Finally, by \eqref{rule_2}, we obtain
\begin{equation}
\Prob[\alpha \to \beta_{\sigma_{i-1}+g_{i}-m+2} = 0]=p_2.\label{transition_eq_5_beta_0}
\end{equation}
Combining \eqref{transition_eq_3_beta_0}, \eqref{transition_eq_4_beta_0} and \eqref{transition_eq_5_beta_0}, we deduce the second identity stated in \eqref{transition_eq_0_strings_beta_all_0}.

Finally, when $\alpha=0^{n}$, the entire configuration $\alpha$ is its own $0$-string, and it is immediate from \eqref{rule_1} that 
\begin{align}
{}&\Prob[\alpha \to \beta_{j}=0]=(1-p_1) \quad \text{for each } j\in[n],\nonumber
\end{align}
which yields $\Prob\left[\alpha\rightarrow\beta\right]=\Prob\left[0^n \rightarrow 0^n\right]=(1-p_1)^{n}$. This completes the proof of \cref{lem:transition_probabilities_beta_all_0}.
\end{proof}
 
Note that when $\beta=0^{n}$ and $\alpha\neq 0^{n}$, we may express the findings of \cref{lem:transition_probabilities_beta_all_0} as follows:
\begin{equation}\label{transition_eq_13}
\Prob\left[\alpha\rightarrow\beta\right]=\Prob\left[\alpha\rightarrow 0^{n}\right]=\prod_{i\in[S]}\left\{(1-p_{1})^{g_{i}-m+1}p_{2}\right\}\chi\left[g_{i}\geqslant(m-1)\right].
\end{equation}
Our task, now, is to express \eqref{transition_eq_13} in terms of $N_{1}(\alpha)$, $N_{10^{r}1}(\alpha)$ for each $r \in [m-2]$, and $N_{0^{m-1}1}(\alpha)$ (as defined in \eqref{notation:1}, \eqref{notation:10^{r}1} and \eqref{notation:0^{r}1}). 
We note that each $0$-string of $\alpha$ with length $g_{i}\geqslant(m-1)$ corresponds to an occurrence of the subsequence $0^{m-1}1$ in $\alpha$, whereas each $0$-string of $\alpha$ with length $g_{i}=r$ corresponds to an occurrence of the subsequence $10^{r}1$ in $\alpha$, for each $r \in [m-2]$. Consequently, we have
\begin{equation}
\sum_{i \in [S]}\chi\left[g_{i}\geqslant (m-1)\right]=N_{0^{m-1}1}(\alpha) \quad \text{and} \quad \sum_{i\in[R]}\chi\left(g_{i}=s\right)=N_{10^{s}1}(\alpha) \text{ for each } s\in[m-2].\nonumber
\end{equation}
Implementing the observations made above, we see that the exponent of $(1-p_1)$ in \eqref{transition_eq_13} equals:
\begin{align}
\sum_{i\in [S]}(g_{i}-m+1)\chi\left(g_{i}\geqslant m-1\right)=n-N_{1}(\alpha)-\sum_{r\in[m-2]}r N_{10^{r}1}(\alpha)-(m-1)N_{0^{m-1}1}(\alpha).\nonumber
\end{align}
The exponent of $p_2$ in \eqref{transition_eq_13} equals $N_{0^{m-1}1}(\alpha)$. Consequently, we can write \eqref{transition_eq_13} as 
\begin{equation}
\Prob[\alpha \rightarrow \beta]=(1-p_1)^{n-N_{1}(\alpha)-\sum_{s\in[m-2]}s N_{10^{s}1}(\alpha)-(m-1)N_{0^{m-1}1}(\alpha)}p_2^{N_{0^{m-1}1}(\alpha)}.\label{transition_eq_beta_all_0}
\end{equation}

On the other hand, when $\alpha=\beta=0^{n}$, each of $N_{1}(\alpha)$, $N_{10^{s}1}(\alpha)$ for $s \in [m-2]$, and $N_{0^{m-1}1}(\alpha)$ equals $0$. Consequently, the expression in \eqref{transition_eq_beta_all_0} boils down to simply $(1-p_1)^{n}$, which matches with what we have deduced in \cref{lem:transition_probabilities_beta_all_0}. 

Substituting \eqref{transition_eq_beta_all_0} and \eqref{limiting_distribution} into the left side of \eqref{gen_identity}, we obtain:
\begin{align*}
\sum_{\alpha \in \Omega}\Prob[\alpha \rightarrow \beta]
\pi(\alpha)={}&\sum_{\alpha \in \Omega}(1-p_1)^{n-N_{1}(\alpha)-\sum_{s\in[m-2]}s N_{10^{s}1}(\alpha)-(m-1)N_{0^{m-1}1}(\alpha)}p_2^{N_{0^{m-1}1}(\alpha)}\nonumber\\
&\times \frac{1}{Z_{n,m}}p_1^{N_{1}(\alpha)}(1-p_1)^{\sum_{s=1}^{m-2}s N_{10^{s}1}(\alpha)+(m-1)N_{0^{m-1}1}(\alpha)}p_2^{-N_{0^{m-1}1}(\alpha)}\nonumber\\
={}& \frac{(1-p_1)^{n}}{Z_{n,m}} \sum_{\alpha \in \Omega}\left(\frac{p_1}{1-p_1}\right)^{N_{1}(\alpha)}.
\end{align*}
We can refine the sum according to the number of $1$'s in $\alpha$ to obtain
\begin{align*}
\sum_{\alpha \in \Omega}\Prob[\alpha \rightarrow \beta] \pi(\alpha)
=& \frac{(1-p_1)^{n}}{Z_{n,m}} \sum_{r=0}^{n}\sum_{\substack{\alpha \in \Omega \\ N_{1}(\alpha)=r}}
\left(\frac{p_1}{1-p_1}\right)^{r}\nonumber\\
={}&\frac{(1-p_1)^{n}}{Z_{n,m}}\sum_{r=0}^{n}\binom{n }{ r}\left(\frac{p_1}{1-p_1}\right)^{r}
=\frac{(1-p_1)^{n}}{Z_{n,m}}\left(1+\frac{p_1}{1-p_1}\right)^{n}
=\frac{1}{Z_{n,m}},\nonumber
\end{align*}
by the binomial theorem, which agrees with the expression for $\pi(0^n)$ as given by \eqref{limiting_distribution}. 

\subsection{The case of $\beta\neq 0^{n}$}\label{subsec:beta_not_all_0} Recall the first block $\beta_{[k+\ell]}$ of $\beta$, when $\beta\neq 0^{n}$, as described in \eqref{beta_first_block}. 
\begin{lemma}\label{lem:transition_probabilities_first_k_coordinates} 
For any $\beta \neq 0^{n}$, the transition probability $\Prob[\alpha\rightarrow\beta_{[k]}]$ equals
\begin{equation}
\begin{cases}
p_{1}^{k} & \text{when } \alpha_{k+m-1}=0,\\
p_1^{k-1}(1-p_2) & \text{ when } \alpha_{k+m-1}=1.
  \end{cases}
\end{equation}
\end{lemma}
\begin{proof}
To begin with, we note that the assumption $\beta \neq 0^{n}$ ensures that $k=k_{1}\geqslant 1$ in \eqref{beta_decomposed}. Since we focus on $\alpha \in \mathcal{S}_{\beta}$, each term of the product appearing in \eqref{transition_probability_decomposed} must be strictly positive. In particular, we must ensure that
\begin{enumerate}
\item $\Prob[\alpha\rightarrow\beta_{[k]}]$ is strictly positive, which, in turn, is equivalent to having each term of the product in \eqref{partial_transition_probability}, with $j_{1}=0$ and $j_{2}=k$, strictly positive, and
\item the probability $\Prob[\alpha \to \beta_{k+\ell+1} = 1]$ is strictly positive as well, which is an additional requirement unless $\beta$ comprises a single block.
\end{enumerate}
These requirements, along with \eqref{necessary_cond_for_1} and the fact that $k\geqslant 1$, yield
\begin{equation}\label{alpha_coordinates_surely_0}
\alpha_{i}=0 \text{ for each } i \in \{1,2,\ldots,k,k+1,\ldots,k+m-2\} \cup \{k+\ell+1,\ldots,k+\ell+m-1\}
\end{equation} 
whenever $\alpha \in \mathcal{S}_{\beta}$. From \eqref{alpha_coordinates_surely_0} and \eqref{rule_1}, we have:  
\begin{equation}
\Prob[\alpha \to \beta_j = 1]
=p_1 \quad \text{for each } j\in[k-1].\label{transition_1_to_k-1}
\end{equation}
Due to \eqref{alpha_coordinates_surely_0} and since $\alpha_{k+m-1}$ could be either $0$ or $1$ when each of $p_1$ and $p_2$ is in $(0,1)$, we have
\begin{equation}\label{transition_k}
\Prob[\alpha \to \beta_k = 1]
= \begin{cases} 
   p_1 & \text{when } \alpha_{k+m-1}=0,\\
   1-p_2 & \text{ when } \alpha_{k+m-1}=1.
  \end{cases}
\end{equation}
Combining these findings, we obtain the desired result. We refer the reader to \cref{fig_1} for an illustration.
\end{proof}

As a corollary to \cref{lem:transition_probabilities_first_k_coordinates}, we obtain:
\begin{corollary}\label{subsec:ell_less_equal_m-2}
When $\beta\neq 0^{n}$ and $\ell \leqslant (m-2)$, we have $\Prob[\alpha\rightarrow\beta_{[k+\ell]}] = p_1^{k}(1-p_1)^{\ell}$.
\end{corollary}
\begin{proof}
Since $\ell \leqslant (m-2)$, we have $(k+\ell+1)\leqslant (k+m-1)$, which, along with \eqref{alpha_coordinates_surely_0}, yields \begin{equation}
\alpha_{i}=0 \text{ for each } i\in\{1,2,\ldots,k,k+1,\ldots,k+m-2,k+m-1,\ldots,k+\ell+m-1\}.\nonumber
\end{equation} 
for any $\alpha\in\mathcal{S}_{\beta}$. As shown in \cref{lem:transition_probabilities_first_k_coordinates}, we have $\Prob[\alpha \to \beta_{j}=1] = p_{1}$ for each $j \in [k]$, while by \eqref{rule_1}, we have $\Prob[\alpha \to \beta_j = 0] = (1-p_1)$ for each $j \in [k+1,k+\ell]$. Substituting these in the expression appearing in \eqref{partial_transition_probability} for $j_{2}=s_{1}=(k+\ell)$ and $j_{1}=0$, we obtain $\Prob[\alpha\rightarrow\beta_{[k+\ell]}] = p_1^{k}(1-p_1)^{\ell}$.
\end{proof}

\subsubsection{A decomposition of $\alpha_{[k+1,k+\ell]}$ when $\ell \geqslant (m-1)$}
\label{subsec:ell_greater_m-2_beta_not_0} 
For any $\alpha \in \mathcal{S}_{\beta}$, we focus on $\alpha_{[k+1,k+\ell]}$.
From \eqref{alpha_coordinates_surely_0}, we have $\alpha_{i}=0$ for each $i \in [k+m-2]$.
This has been illustrated, for the reader's convenience, in \cref{fig_1}. 

\begin{figure}[h!]
  \centering
\fbox{    \includegraphics[width=0.8\textwidth]{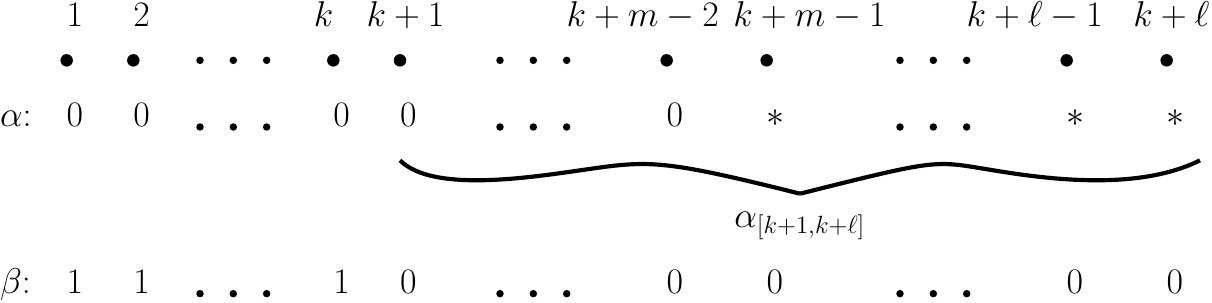}}
\caption{The part, $\alpha_{[k+1,k+\ell]}$, of $\alpha$ that we henceforth focus on}
  \label{fig_1}
\end{figure}

We now represent $\alpha_{[k+1,k+\ell]}$ as
\begin{equation}\label{alpha'_decomposed}
\alpha_{[k+1,k+\ell]}=0^{g_{1}}1^{h_{1}}0^{g_{2}}1^{h_{2}}\ldots 0^{g_{S}}1^{h_{S}}0^{g_{S+1}},
\end{equation}
where $S \in \mathbb{N}_{0}$, $g_{i} \in \mathbb{N}_{0}$ for each $i \in [S+1]$ and $h_{i} \in \mathbb{N}_{0}$ for each $i \in [S]$, and the following criteria are satisfied:
\begin{enumerate*}
\item $g_{1}\geqslant (m-2)$, and
\item $\sum_{i \in [S]} \left(g_{i}+h_{i}\right)+g_{S+1}=\ell$.
\end{enumerate*}
Note that if $\alpha_{[k+1,k+\ell]}=0^{\ell}$, we set $S=0$, so that $g_{1}=g_{S+1}=\ell$, whereas if $\alpha_{[k+1,k+\ell]}=0^{m-2}1^{\ell-m+2}$, we set $S=1$, $g_{1}=(m-2)$, $h_{1}=(\ell-m+2)$ and $g_{S+1}=g_{2}=0$. Evidently, the inequality $g_{S+1}>0$ implies that $\alpha_{[k+1,k+\ell]}$ ends in a $0$, whereas $h_{S}>0$ and $g_{S+1}=0$ together imply that $\alpha_{[k+1,k+\ell]}$ ends in a $1$. An example of the decomposition of $\alpha_{[k+1,k+\ell]}$, as given by \eqref{alpha'_decomposed}, has been illustrated in \cref{fig_2}.

\begin{figure}[h!]
  \centering
\fbox{    \includegraphics[width=0.8\textwidth]{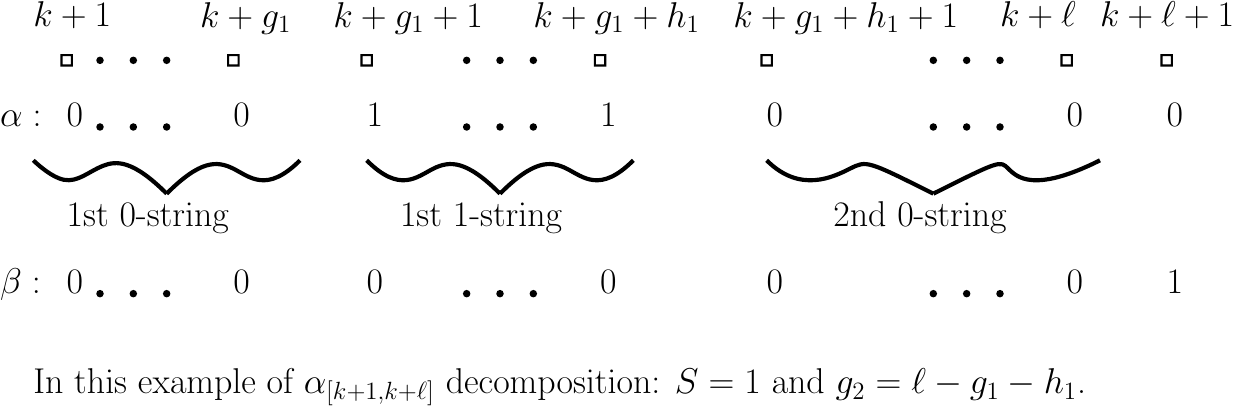}}
\caption{An example decomposition of $\alpha_{[k+1,k+\ell]}$ as given by \eqref{alpha'_decomposed}}
  \label{fig_2}
\end{figure}

It is immediate, from \eqref{alpha'_decomposed}, that $S$ equals $0$ iff $\alpha_{[k+1,k+\ell]}=0^{\ell}$. Similar to \eqref{sigma_i_defns_beta_all_0}, we set
\begin{equation}
\sigma_{0}=k \text{ and } 
\sigma_{i}=\sigma_{i-1}+\left(g_{i}+h_{i}\right) 
\text{ for each } i \in [S].\label{s'_defns}
\end{equation}
As before, we refer to the tuple $\left(\sigma_{i-1}+1,\sigma_{i-1}+2,\ldots,\sigma_{i-1}+g_{i}\right)$ as the $i$-th $0$-string of $\alpha_{[k+1,k+\ell]}$, and the length of this string equals $g_{i}$, for $i\in[S+1]$. 
Likewise, if $\alpha_{[k+1,k+\ell]}\neq 0^{\ell}$ (in which case $h_{i}\geqslant 1$ for each $i \in [S]$), we refer to $\left(\sigma_{i-1}+g_{i}+1,\ldots,\sigma_{i}\right)$ as the $i$-th $1$-string of $\alpha_{[k+1,k+\ell]}$, for each $i \in [S]$, and its length equals $h_{i}$. 

The lemma that follows summarizes the transition probabilities $\Prob\left[\alpha\to\beta_{j}\right]$ as the coordinate $j$ is allowed to vary over the various $1$-strings and $0$-strings of $\alpha_{[k+1,k+\ell]}$.
\begin{lemma}\label{lem:transition_probabilities_alpha_{[k+1,k+ell]}}
The transition probabilities corresponding to the $1$-strings of $\alpha_{[k+1,k+\ell]}$ are given by
\begin{equation}\label{transition_eq_1_strings}
\prod_{j\in\left[\sigma_{i-1}+g_{i}+1,\sigma_{i}\right]}\Prob[\alpha \to \beta_j=0]=1 \text{ for } i \in [S].
\end{equation}
On the other hand, when it comes to the $0$-strings of $\alpha_{[k+1,k+\ell]}$, we have:
\begin{equation}\label{transition_eq_0_strings}
\prod_{j\in\left[\sigma_{i-1}+1,\sigma_{i-1}+g_{i}\right]}\Prob[\alpha \to \beta_j = 0]=
\begin{cases}
1 &\text{ when } i\in[S] \text{ and } g_{i}\leqslant(m-2),\\
(1-p_1)^{g_{i}-m+1}p_2 &\text{ when } i\in[S] \text{ and } g_{i}\geqslant(m-1),\\
(1-p_1)^{g_{S+1}} &\text{ when } i=(S+1).
\end{cases}
\end{equation} 
\end{lemma}
\begin{proof}
When $\alpha_{[k+1,k+\ell]} \neq 0^{\ell}$, we focus on each $1$-string of $\alpha_{[k+1,k+\ell]}$, and deduce \eqref{transition_eq_1_strings}, using \eqref{necessary_cond_for_1}, the same way as we deduced \eqref{transition_eq_1_strings_beta_all_0}. 

Recall, from the discussion following \eqref{alpha'_decomposed}, that when $\alpha_{[k+1,k+\ell]}=0^{\ell}$, we have $S=0$, implying $[S]=\emptyset$, so that we need not consider $0$-strings belonging to the first two categories mentioned in \eqref{transition_eq_0_strings}. Moreover, the first two identities stated in \eqref{transition_eq_0_strings} can be established in exactly the same way as we have established the identities in \eqref{transition_eq_0_strings_beta_all_0}.

Finally, we consider the $0$-string $\left(\sigma_{S}+1,\sigma_{S}+2,\ldots,\sigma_{S}+g_{S+1}\right)$ that $\alpha_{[k+1,k+\ell]}$ ends with (if it at all ends in a $0$-string). From \eqref{alpha_coordinates_surely_0}, we have $\alpha_{i}=0$ for each $i \in [k+\ell+1,k+\ell+m-1]$. If now, in addition, we know that $\alpha_{[k+1,k+\ell]}$ ends in a $0$-string, so that $\alpha_{\sigma_{S}+1}=\alpha_{\sigma_{S}+2}=\cdots=\alpha_{\sigma_{S}+g_{S+1}} (=\alpha_{k+\ell})=0$, we have 
\begin{equation}
\alpha_{\sigma_{S}+1}=\alpha_{\sigma_{S}+2}=\cdots=\alpha_{\sigma_{S}+g_{S+1}}(=\alpha_{k+\ell})=\alpha_{k+\ell+1}=\alpha_{k+\ell+2}=\cdots=\alpha_{k+\ell+m-1}=0.\nonumber
\end{equation}
This yields, via \eqref{rule_1}: 
\begin{align}
\Prob[\alpha \to \beta_j = 0] = (1-p_1) \text{ for each }j \in \left\{\sigma_{S}+1,\sigma_{S}+2,\ldots,\sigma_{S}+g_{S+1}=k+\ell\right\},\nonumber
\end{align}
and upon taking the product over all such $j$, we obtain the third identity of \eqref{transition_eq_0_strings}. 

In order to illustrate how the second and third identities of \eqref{transition_eq_0_strings} work, we refer the reader to \cref{fig_3} and  \cref{fig_4} respectively (in each of these, the arrows indicate cell-wise transition probabilities).
\end{proof}

\begin{figure}[h!]
  \centering
\fbox{    \includegraphics[width=0.8\textwidth]{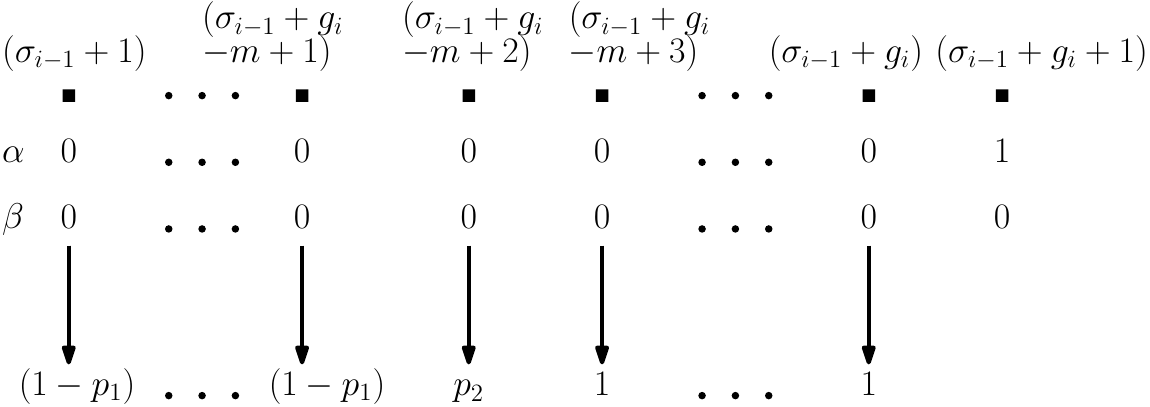}}
\caption{The transition probabilities arising out of a $0$-th string of $\alpha_{[k+1,k+\ell]}$, with length $g_{i}\geqslant(m-1)$, that $\alpha_{[k+1,k+\ell]}$ does not end with}
  \label{fig_3}
\end{figure}

\begin{corollary}\label{cor:0_length_first_block_geq_m-1}
When $\beta\neq 0^{n}$ and $\ell \geqslant (m-1)$, we have 
\begin{align}
\Prob[\alpha\rightarrow\beta_{[k+\ell]}]
= & p_1^{k-1}\left\{p_1\chi\left[g_{1}\geqslant m-1\right]+(1-p_2)\chi\left[g_{1}= m-2\right]\right\}\nonumber\\
& \times \prod_{i \in [S]} \chi\left[g_{i}\geqslant m-1\right]
\left\{(1-p_1)^{g_{i}-m+1}p_2\right\}(1-p_1)^{g_{S+1}}.\label{final_transition_first_block_case_1}
\end{align}   
\end{corollary}
\begin{proof}
We begin by noting that $g_{1}\geqslant (m-1)$ implies $\alpha_{k+m-1}=0$, whereas $g_{1}=(m-2)$ implies that $\alpha_{k+m-1}=1$. Keeping this in mind, and combining the findings deduced in \cref{lem:transition_probabilities_first_k_coordinates} and \cref{lem:transition_probabilities_alpha_{[k+1,k+ell]}}, we obtain the expression given in \eqref{final_transition_first_block_case_1} for the transition probability $\Prob[\alpha\rightarrow\beta_{[k+\ell]}]$ whenever $\ell \geqslant (m-1)$.
\end{proof}

\begin{figure}[h!]
  \centering
\fbox{    \includegraphics[width=0.8\textwidth]{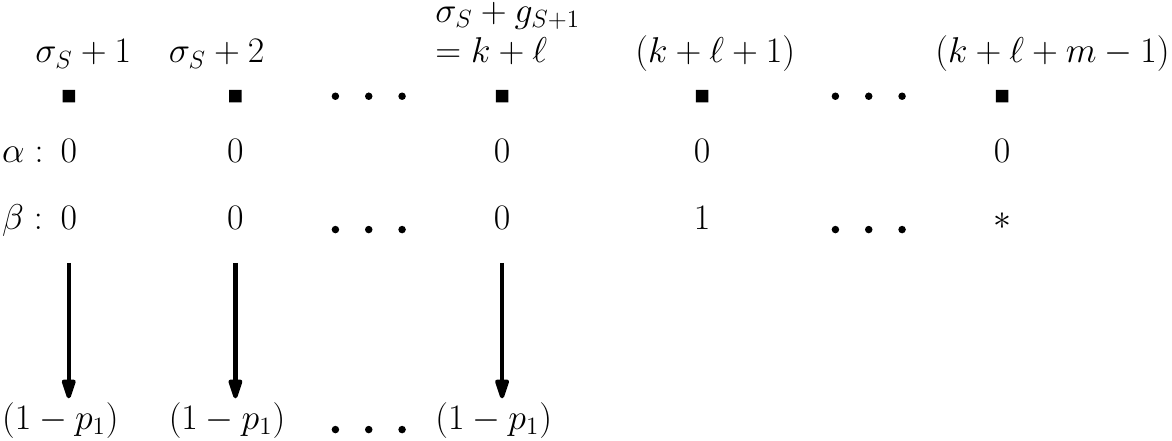}}
\caption{Transition probabilities arising out of the $0$-string of $\alpha_{[k+1,k+\ell]}$ that $\alpha_{[k+1,k+\ell]}$ ends with (if such a $0$-string exists)}
  \label{fig_4}
\end{figure}

\subsubsection{The final expression for $\Prob[\alpha\rightarrow\beta_{[k+\ell]}]$}
\label{subsec:express} 
Similar to the quantities introduced in \eqref{notation:1}, \eqref{notation:10^{r}1} and \eqref{notation:0^{r}1}, we may define, for any configuration $\alpha\in\Omega$, and any $j_{1}\in[n]\cup\{0\}$ and $j_{2}\in[n]$ with $j_{1}<j_{2}$, the following (recalling the notation introduced in \eqref{partial_configuration}):
\begin{enumerate}
\item the number of occurrences of the symbol $1$ in the tuple $\alpha_{[j_{1}+1,j_{2}]}$ is given by
\begin{equation}
N_{1}\left(\alpha_{[j_{1}+1,j_{2}]}\right)=\sum_{j\in[j_{1}+1,j_{2}]}\chi\left[\alpha_{j}=1\right];\nonumber
\end{equation}
\item the number of occurrences of the subsequence $10^{r}1$, for any $r\in\mathbb{N}$, in the tuple $\alpha_{[j_{1}+1,j_{2}]}$ is given by
\begin{equation}
N_{10^{r}1}\left(\alpha_{[j_{1}+1,j_{2}]}\right)=\sum_{j\in[j_{1}+1,j_{2}-r-1]}\chi\left[\left(\alpha_{j},\alpha_{j+1},\ldots,\alpha_{j+r+1}\right)=10^{r}1\right];\nonumber
\end{equation}
\item the number of occurrences of the subsequence $0^{r}1$, for any $r\in\mathbb{N}$, in the tuple $\alpha_{[j_{1}+1,j_{2}]}$, is given by
\begin{equation}
N_{0^{r}1}\left(\alpha_{[j_{1}+1,j_{2}]}\right)=\sum_{j\in[j_{1}+1,j_{2}-r]}\chi\left[\left(\alpha_{j},\alpha_{j+1},\ldots,\alpha_{j+r}\right)=0^{r}1\right].\nonumber 
\end{equation}
\end{enumerate} 
As before, we abbreviate $\alpha_{[j_{1}+1,j_{2}]}$ further to simply $\alpha_{[j_{2}]}$ in each of the notations above when $j_{1}=0$.

Our goal, now, is to express the probabilities obtained in \cref{subsec:ell_less_equal_m-2} and \cref{cor:0_length_first_block_geq_m-1} uniformly in terms of $N_{1}\left(\alpha_{[k+\ell]}\right)$, $N_{10^{r}1}\left(\alpha_{[k+\ell]}\right)$ for $r\in[m-2]$, and $N_{0^{m-1}1}\left(\alpha_{[k+\ell]}\right)$. To this end, we define, for each $i\in[R]$, recalling the decomposition in \eqref{beta_decomposed}, and the definitions in \eqref{s_{i}_defns}:
\begin{equation}
\kappa_{i}=\ell_{i}-N_{1}\left(\alpha_{[s_{i-1}+1,s_{i}]}\right)-\sum_{r\in[m-2]}rN_{10^{r}1}\left(\alpha_{[s_{i-1}+1,s_{i}]}\right)-(m-1)N_{0^{m-1}1}\left(\alpha_{[s_{i-1}+1,s_{i}]}\right).\label{kappa_{i}_defn}
\end{equation}
We now define 
\begin{align}\label{kappa_defn}
\kappa={}&\sum_{i\in[R]}\kappa_{i}=\sum_{i \in [R]}\ell_{i}-N_{1}(\alpha)-\sum_{r\in[m-2]}r N_{10^{r}1}(\alpha)-(m-1)N_{0^{m-1}1}(\alpha)\nonumber\\
={}&n-N_{1}(\beta)-N_{1}(\alpha)-\sum_{r\in[m-2]}r N_{10^{r}1}(\alpha)-(m-1)N_{0^{m-1}1}(\alpha),
\end{align}
since $N_{1}(\beta)=\sum_{i\in[R]}k_{i}$ from \eqref{beta_decomposed}.

Recall, from \eqref{beta_first_block} and the paragraph preceding it, that we replace $k_{1}$ by $k$ and $\ell_{1}$ by $\ell$ when focusing on the first block of $\beta$ (as given by the decomposition in \eqref{beta_decomposed}). With this altered notation in mind, we state and prove the following lemma:
\begin{lemma}
\label{lem:identity}
Whenever $k, \ell \geqslant 1$, we have, with $\kappa_{1}$ as defined in \eqref{kappa_{i}_defn}:
\begin{align}
\Prob[\alpha\rightarrow\beta_{[k+\ell]}]={}&p_1^{k}(1-p_1)^{\kappa_{1}}p_2^{N_{0^{m-1}1}\left(\alpha_{[k+\ell]}\right)}\Bigg\{\frac{(1-p_1)(1-p_2)}{p_1p_2}\chi\left[\ell\geqslant m-1,\alpha_{k+m-1}=1\right]\nonumber\\&+\chi\left[\ell\geqslant m-1,\alpha_{k+m-1}=0\right]+\chi\left[\ell\leqslant m-2\right]\Bigg\}.\label{eq:first_block_transition_probability_final}
\end{align}
\end{lemma}
\begin{proof}
To prove this, we consider the following possible cases:
\begin{enumerate}[label=(A\arabic*), ref=A\arabic*]
\item \label{A1} $\ell\leqslant (m-2)$,
\item \label{A2} $\ell \geqslant (m-1)$ and $g_{1}=(m-2)$ (the second criterion is equivalent to $\alpha_{k+m-1}=1$), and 
\item \label{A3} $\ell\geqslant (m-1)$ and $g_{1}\geqslant (m-1)$ (the second criterion is equivalent to $\alpha_{k+m-1}=0$).
\end{enumerate}
The proof for each of these cases has been written down separately.

\textbf{Proof of (A1):} From the results in \cref{subsec:ell_less_equal_m-2},
we know that when $\ell\leqslant(m-2)$, we have $\alpha_{1}=\alpha_{2}=\cdots=\alpha_{k+\ell+m-1}=0$. 
Consequently, each of $N_{1}\left(\alpha_{[k+\ell]}\right)$, $N_{10^{r}1}\left(\alpha_{[k+\ell]}\right)$ for all $r\in[m-2]$, and $N_{0^{m-1}1}\left(\alpha_{[k+\ell]}\right)$, equals $0$. 
The expression on the right side of \eqref{eq:first_block_transition_probability_final} thus boils down to simply $p_1^{k}(1-p_1)^{\ell}$, which is exactly what $\Prob[\alpha\rightarrow\beta_{[k+\ell]}]$ was shown to be equal to when $\ell\leqslant(m-2)$ in \cref{subsec:ell_less_equal_m-2}

\textbf{Proof of (A2):} From \eqref{final_transition_first_block_case_1}, we see that, when $\ell \geqslant (m-1)$ and $g_{1}=(m-2)$, we have
\begin{align}
\Prob[\alpha\rightarrow\beta_{[k+\ell]}]={}&p_1^{k-1}(1-p_2)\prod_{i \in [S]}\left\{(1-p_1)^{g_{i}-m+1}p_2\chi\left[g_{i}\geqslant (m-1)\right]\right\}(1-p_1)^{g_{S+1}}.\label{final_transition_first_block_case_1_subcase_1}
\end{align} 
The exponent of $(1-p_1)$ in \eqref{final_transition_first_block_case_1_subcase_1} is given by (using the fact that $g_{1}=(m-2)$): 
\begin{align}
{}&\sum_{i\in[S]}\left(g_{i}-m+1\right)\chi\left[g_{i}\geqslant(m-1)\right]+g_{S+1}=\sum_{i\in[S]\setminus\{1\}}\left(g_{i}-m+1\right)\chi\left[g_{i}\geqslant(m-1)\right]+g_{S+1}\nonumber\\
={}&\sum_{i\in[S]\setminus\{1\}}g_{i}\chi\left[g_{i}\geqslant(m-1)\right]-(m-1)\sum_{i\in[S]\setminus\{1\}}\chi\left[g_{i}\geqslant(m-1)\right]+g_{S+1}\nonumber\\
={}&\sum_{i\in[S]}g_{i}-\sum_{i\in[S]}g_{i}\chi\left[g_{i}\leqslant(m-2)\right]-(m-1)\sum_{i\in[S]\setminus\{1\}}\chi\left[g_{i}\geqslant(m-1)\right]+g_{S+1}\nonumber\\
={}&\sum_{i\in [S+1]}g_{i}-\sum_{i\in[S]}g_{i}\chi\left[g_{i}\leqslant(m-2)\right]-(m-1)\sum_{i\in[S]\setminus\{1\}}\chi\left[g_{i}\geqslant(m-1)\right]\nonumber\\
={}&\ell-N_{1}\left(\alpha_{[k+\ell]}\right)-(m-2)-\sum_{i\in[S]\setminus\{1\}}g_{i}\chi\left[g_{i}\leqslant(m-2)\right]-(m-1)\sum_{i\in[S]\setminus\{1\}}\chi\left[g_{i}\geqslant(m-1)\right],\label{exponent_(1-p_1)_initial}
\end{align}
where the identity $\sum_{i\in [S+1]}g_{i}=\ell-N_{1}\left(\alpha_{[k+\ell]}\right)$ follows from the representation of $\alpha_{[k+1,k+\ell]}$ shown in \eqref{alpha'_decomposed}, and since $\alpha_{1}=\alpha_{2}=\cdots=\alpha_{k}=0$, as stated in \eqref{alpha_coordinates_surely_0}. As explained right after \eqref{alpha'_decomposed}, 
\begin{align}
\left(\alpha_{\sigma_{i}},\alpha_{\sigma_{i}+1},\ldots,\alpha_{\sigma_{i}+g_{i+1}},\alpha_{\sigma_{i}+g_{i+1}+1}\right)=10^{g_{i+1}}1 \quad \text{for each } i\in[S-1],\label{10^{r}1_subsequences_alpha'}
\end{align}
which tells us that the fourth term of \eqref{exponent_(1-p_1)_initial} can be expressed as
\begin{align}
{}&\sum_{i\in[S]\setminus\{1\}}g_{i}\chi\left[g_{i}\leqslant(m-2)\right]=\sum_{r\in[m-2]}\sum_{i\in[S]\setminus\{1\}:g_{i}=r}g_{i}\nonumber\\
={}&\sum_{r\in[m-2]}r N_{10^{r}1}\left(\alpha_{[k+1,k+\ell]}\right)=\sum_{r\in[m-2]}r N_{10^{r}1}\left(\alpha_{[k+\ell]}\right),\label{exponent_(1-p_1)_part_1}
\end{align}
where the last equality follows from observing that, since $\alpha_{1}=\alpha_{2}=\cdots=\alpha_{k}=\alpha_{k+1}=\cdots=\alpha_{k+m-2}=0$ (from \eqref{alpha_coordinates_surely_0}), we have $N_{10^{r}1}\left(\alpha_{[k+1,k+\ell]}\right)=N_{10^{r}1}\left(\alpha_{[k+\ell]}\right)$ for each $r\in[m-2]$. Next, note that since $g_{1}=m-2$, we have $\alpha_{k+m-1}=1$, which, along with \eqref{alpha_coordinates_surely_0}, implies that 
\begin{equation}
\left(\alpha_{k},\alpha_{k+1},\ldots,\alpha_{k+m-2},\alpha_{k+m-1}\right)=0^{m-1}1.\label{extra_0^{m-1}1_subsequence_beginning_of_alpha'}
\end{equation}
Note that this is one of the steps in our argument where we crucially make use of our assumption that $\beta\neq 0^{n}$, as it ensures that $k\geqslant 1$. Next, from \eqref{10^{r}1_subsequences_alpha'}, it is immediate that
\begin{equation}
\left(\alpha_{\sigma_{i}+g_{i+1}-m+2},\ldots,\alpha_{\sigma_{i}+g_{i+1}},\alpha_{\sigma_{i}+g_{i+1}+1}\right)=0^{m-1}1 \quad \text{for each } i\in[S] \text{ with } g_{i+1}\geqslant(m-1).\label{0^{m-1}1_subsequences_alpha'}
\end{equation}
From \eqref{extra_0^{m-1}1_subsequence_beginning_of_alpha'} and \eqref{0^{m-1}1_subsequences_alpha'}, we conclude that 
\begin{equation}
N_{0^{m-1}1}\left(\alpha_{[k+\ell]}\right)=1+\left|\left\{i\in[S]:g_{i+1}\geqslant(m-1)\right\}\right|=1+N_{0^{m-1}1}\left(\alpha_{[k+1,k+\ell]}\right).\label{one_extra_0^{m-1}1}
\end{equation}
Combining the third and fifth terms of \eqref{exponent_(1-p_1)_initial}, and applying the observation in \eqref{one_extra_0^{m-1}1}, we obtain:
\begin{align}
-(m-2)-(m-1)\sum_{i\in[S]\setminus\{1\}}\chi\left[g_{i}\geqslant(m-1)\right]=-(m-1)N_{0^{m-1}1}\left(\alpha_{[k+\ell]}\right)+1.\label{exponent_(1-p_1)_part_2}
\end{align}
Substituting the expressions obtained in \eqref{exponent_(1-p_1)_part_1} and \eqref{exponent_(1-p_1)_part_2}, in \eqref{exponent_(1-p_1)_initial}, the exponent of $(1-p_1)$ in \eqref{final_transition_first_block_case_1_subcase_1} can be rewritten as:
\begin{equation}
\ell-N_{1}\left(\alpha_{[k+\ell]}\right)-\sum_{r\in[m-2]}r N_{10^{r}1}\left(\alpha_{[k+\ell]}\right)-(m-1)N_{0^{m-1}1}\left(\alpha_{[k+\ell]}\right)+1.\label{exponent_(1-p_1)_final}
\end{equation}
Next, the exponent of $p_2$ in \eqref{final_transition_first_block_case_1_subcase_1} equals, by the observations made in \eqref{extra_0^{m-1}1_subsequence_beginning_of_alpha'} and \eqref{0^{m-1}1_subsequences_alpha'}:
\begin{equation}
\sum_{i\in[S]}\chi\left[g_{i}\geqslant(m-1)\right]=N_{0^{m-1}1}\left(\alpha_{[k+\ell]}\right)-1.\label{exponent_p_2_final}
\end{equation}
Incorporating \eqref{exponent_(1-p_1)_final} and \eqref{exponent_p_2_final} into \eqref{final_transition_first_block_case_1_subcase_1}, we see that $\Prob[\alpha\rightarrow\beta_{[k+\ell]}]$ indeed satisfies \eqref{eq:first_block_transition_probability_final} when $\ell\geqslant(m-1)$ and $g_{1}=(m-2)$ (equivalently, $\ell\geqslant(m-1)$ and $\alpha_{k+m-1}=1$). This completes the proof for the case \eqref{A2}.

\textbf{Proof of (A3):} From \eqref{final_transition_first_block_case_1}, we see that when $\ell \geqslant (m-1)$ and $g_{1}\geqslant(m-1)$, we have
\begin{align}
\Prob[\alpha\rightarrow\beta_{[k+\ell]}]={}&p_1^{k}\prod_{i \in [S]}\left\{(1-p_1)^{g_{i}-m+1}p_2\chi\left[g_{i}\geqslant (m-1)\right]\right\}(1-p_1)^{g_{S+1}}.\label{final_transition_first_block_case_1_subcase_2}
\end{align}
The exponent of $(1-p_1)$ in \eqref{final_transition_first_block_case_1_subcase_2} equals (keeping in mind that $g_{1}\geqslant(m-1)$ here):
\begin{align}
{}&\sum_{i\in[S]}\left(g_{i}-m+1\right)\chi\left[g_{i}\geqslant (m-1)\right]+g_{S+1}\nonumber\\
={}&\ell-N_{1}\left(\alpha_{[k+\ell]}\right)-\sum_{i\in[S]\setminus\{1\}}g_{i}\chi\left[g_{i}\leqslant (m-2)\right]-(m-1)\sum_{i\in[S]}\chi\left[g_{i}\geqslant (m-1)\right],\label{exponent_(1-p_1)_initial_1}
\end{align}
with the final expression having been derived the same way as \eqref{exponent_(1-p_1)_initial}. To begin with, we note that, as proved in \eqref{exponent_(1-p_1)_part_1}, here too, we have $\sum_{i\in[S]\setminus\{1\}}g_{i}\chi\left[g_{i}\leqslant (m-2)\right]=\sum_{r\in[m-2]}r N_{10^{r}1}\left(\alpha_{[k+\ell]}\right)$. 

Next, from \eqref{alpha'_decomposed} and \eqref{alpha_coordinates_surely_0}, we get $\alpha_{1}=\alpha_{2}=\cdots=\alpha_{k}=\alpha_{k+1}=\cdots=\alpha_{k+g_{1}}=0$ and $\alpha_{k+g_{1}+1}=1$, so that
\begin{equation}
\left(\alpha_{k+g_{1}-m+2},\alpha_{k+g_{1}-m+3},\ldots,\alpha_{k+g_{1}},\alpha_{k+g_{1}+1}\right)=0^{m-1}1.\nonumber
\end{equation}
Since $g_{1}\geqslant (m-1)$, it implies that $k+g_{1}-m+2\geqslant k+1$, so that the above-mentioned occurrence of the subsequence $0^{m-1}1$ is contained entirely within $\alpha_{[k+1,k+\ell]}$ (unlike the occurrence of the subsequence $0^{m-1}1$ mentioned in \eqref{extra_0^{m-1}1_subsequence_beginning_of_alpha'} when $\ell\geqslant(m-1)$ and $g_{1}=(m-2)$). This observation, along with the observation made in \eqref{0^{m-1}1_subsequences_alpha'} and the fact that $\alpha_{1}=\alpha_{2}=\cdots=\alpha_{k}=0$ (from \eqref{alpha_coordinates_surely_0}), yields 
\begin{equation}
N_{0^{m-1}1}\left(\alpha_{[k+\ell]}\right)=N_{0^{m-1}1}\left(\alpha_{[k+1,k+\ell]}\right)=\left|\left\{i\in[S]:g_{i}\geqslant(m-1)\right\}\right|\label{0^{m-1}1_count_alpha_{[k+ell]}_no_extra}
\end{equation}
when $\ell\geqslant (m-1)$ and $g_{1}\geqslant(m-1)$. Incorporating these observations into \eqref{exponent_(1-p_1)_initial_1}, the exponent of $(1-p_1)$ in \eqref{final_transition_first_block_case_1_subcase_2} can be represented as follows:
\begin{equation}
\ell-N_{1}\left(\alpha_{[k+\ell]}\right)-\sum_{r\in[m-2]}r N_{10^{r}1}\left(\alpha_{[k+\ell]}\right)-(m-1)N_{0^{m-1}1}\left(\alpha_{[k+\ell]}\right).\label{exponent_(1-p_1)_final_1}
\end{equation}

The exponent of $p_2$ in \eqref{final_transition_first_block_case_1_subcase_2} equals, via \eqref{0^{m-1}1_count_alpha_{[k+ell]}_no_extra}, 
\begin{equation}
\sum_{i\in[S]}\chi\left[g_{i}\geqslant(m-1)\right]=N_{0^{m-1}1}\left(\alpha_{[k+\ell]}\right).\label{exponent_p_2_final_1}
\end{equation}
Incorporating \eqref{exponent_(1-p_1)_final_1} and \eqref{exponent_p_2_final_1} into \eqref{final_transition_first_block_case_1_subcase_2}, we see that $\Prob[\alpha\rightarrow\beta_{[k+\ell]}]$ indeed satisfies \eqref{eq:first_block_transition_probability_final} when $\ell\geqslant(m-1)$ and $g_{1}\geqslant(m-1)$ (equivalently, $\ell\geqslant(m-1)$ and $\alpha_{k+m-1}=0$). This completes the proof for the case \eqref{A3}.

Since we have written down the proof in all three cases, we have completed the proof of \cref{lem:identity}.
\end{proof}

\subsubsection{Verification of the master equation}\label{subsec:verification_master_eq} 
We can extend \cref{lem:identity} to 
$\Prob[\alpha\rightarrow\beta_{[s_{i-1}+1,s_{i}]}]$ for all $i\in[R]\setminus\{1\}$ (i.e.\ to all other blocks of $\beta$). 
In other words, we can write, for each $i\in[R]$, recalling $\kappa_{i}$ from \eqref{kappa_{i}_defn},
\begin{align}
\Prob[\alpha\rightarrow\beta_{[s_{i-1}+1,s_{i}]}]={}&p_1^{k_{i}}(1-p_1)^{\kappa_{i}}p_2^{N_{0^{m-1}1}\left(\alpha_{[s_{i-1}+1,s_{i}]}\right)}\Bigg\{\frac{(1-p_1)(1-p_2)}{p_1p_2}\chi\big[\ell_{i}\geqslant m-1,\alpha_{s_{i-1}+k_{i}+m-1}\nonumber\\
&=1\big]+\chi\left[\ell_{i}\geqslant m-1,\alpha_{s_{i-1}+k_{i}+m-1}=0\right]+\chi\left[\ell_{i}\leqslant m-2\right]\Bigg\}.\label{eq:ith_block_transition_probability_final}
\end{align}
Substituting the expression from \eqref{eq:ith_block_transition_probability_final} in the product in \eqref{transition_probability_decomposed}, and recalling $\kappa$ from \eqref{kappa_defn}, we obtain:
\begin{align}
\Prob[\alpha \rightarrow \beta]
={}&\left(\prod_{\substack{i \in [R]:\ell_{i} \geqslant m-1,\\\alpha_{s_{i-1}+k_{i}+m-1}=1}}
\frac{(1-p_1)(1-p_2)}{p_1p_2}\right)
p_1^{\sum_{i\in[R]}k_{i}}
p_2^{\sum_{i\in[R]}N_{0^{m-1}1}\left(\alpha_{[s_{i-1}+1,s_{i}]}\right)} 
(1-p_1)^{\kappa}\nonumber\\
={}&\left(\frac{(1-p_1)(1-p_2)}{p_1p_2}\right)^{\left|\left\{i \in [R]:\ell_{i}\geqslant m-1,\alpha_{s_{i-1}+k_{i}+m-1}=1\right\}\right|}p_1^{N_{1}(\beta)}p_2^{N_{0^{m-1}1}(\alpha)}(1-p_1)^{\kappa},\label{final_transition_probability_eq_beta_not_0}
\end{align}
where $\kappa$ is as defined in \eqref{kappa_defn}, and we use the identity $N_{1}(\beta)=\sum_{i\in[R]}k_{i}$. 
Substituting the expression for $\Prob[\alpha \rightarrow \beta]$ obtained in \eqref{final_transition_probability_eq_beta_not_0}, and the expression for $\pi(\alpha)$ as given by \eqref{limiting_distribution}, in the left side of \eqref{gen_identity}, we obtain \eqref{gen_identity_simplified_step} after simplification. This completes the proof of \cref{lem:gen identity}.

\subsection{Completion of the proof}
\label{subsec:complete}

\begin{proof}[Proof of \cref{thm:main_1}]
The idea, now, is to partition the set $\mathcal{S}_{\beta}$ into suitable equivalence classes and then split the sum in \eqref{gen_identity_simplified_step} into sums over these equivalence classes. Given configurations $\alpha$ and $\alpha'$ in $\mathcal{S}_{\beta}$, we say that $\alpha \sim \alpha'$ if 
\begin{equation}
\alpha_{j}=\alpha'_{j} \text{ for each } j \in 
\left\{s_{i-1}+k_{i}+m, \ldots, s_{i}-1 \right\}, \text{ for each } i \in [R] \text{ with } \ell_{i}> (m-1).\label{equiv_rel}
\end{equation}
Note that for any $\alpha \in \mathcal{S}_{\beta}$ and any $i \in [R]$, we have $\alpha_{s_{i-1}+1}=\alpha_{s_{i-1}+2}=\cdots=\alpha_{s_{i-1}+k_{i}}=\alpha_{s_{i-1}+k_{i}+1}=\cdots=\alpha_{s_{i-1}+k_{i}+m-2}=0$, which follows from \eqref{necessary_cond_for_1} and \eqref{beta_decomposed}. Consequently, for $\alpha, \alpha' \in \mathcal{S}_{\beta}$ with $\alpha \sim \alpha'$, the only cells in which the values of $\alpha$ and $\alpha'$ may differ are those indexed by $(s_{i-1}+k_{i}+m-1)$ for $i \in [R]$ with $\ell_{i}\geqslant(m-1)$. We call $\gamma \in \mathcal{S}_{\beta}$ a \emph{representative configuration} if 
\begin{equation}
\gamma_{s_{i-1}+k_{i}+m-1}=0 \text{ for each } i \in [R] \text{ with } \ell_{i}\geqslant(m-1).\label{representative_defn}
\end{equation}
The set $\mathcal{S}_{\beta}$ is partitioned into equivalence classes by the equivalence relation $\sim$, and it is evident that each such equivalence class contains \emph{precisely one} representative configuration. 

Suppose $\alpha$ belongs to the same equivalence class as a representative configuration $\gamma$, i.e.\ $\alpha \sim \gamma$.
Then, the only extra $1$s in $\alpha$ compared to $\gamma$ are those in the
sites $s_{i-1} + k_i + m-1$ for each $i \in [R]$ with $\ell_i > m-1$. Therefore,
\begin{align}
{}&\left(\frac{(1-p_1)(1-p_2)}{p_1p_2}\right)^{\left|\left\{i \in [R]:\ell_{i}\geqslant m-1,\alpha_{s_{i-1}+k_{i}+m-1}=1\right\}\right|}\left(\frac{p_1}{1-p_1}\right)^{N_{1}(\alpha)}\nonumber\\
={}&\left(\frac{1-p_2}{p_2}\right)^{\left|\left\{i \in [R]:\ell_{i}\geqslant m-1,\alpha_{s_{i-1}+k_{i}+m-1}=1\right\}\right|}
\left( \frac{p_1}{1-p_1} \right)^{N_{1}(\gamma)}.\label{replace_by_representative}
\end{align}

Letting $\Sigma_{\beta}$ denote the set of all representative configurations, and incorporating \eqref{replace_by_representative} into \eqref{gen_identity_simplified_step}, we can write:
\begin{align*}
&\sum_{\alpha \in\mathcal{S}_{\beta}}
\Prob[\alpha \rightarrow \beta] \pi(\alpha)
=\sum_{\gamma\in\Sigma_{\beta}}
\sum_{\substack{\alpha\in\mathcal{S}_{\beta} \\ \alpha\sim\gamma}}\Prob[\alpha\rightarrow\beta] \pi(\alpha)\nonumber\\
={}&\frac{p_1^{N_{1}(\beta)}(1-p_1)^{n-N_{1}(\beta)}}{Z_{n,m}}\sum_{\gamma\in\Sigma_{\beta}}
\left( \frac{p_1}{1-p_1} \right)^{N_{1}(\gamma)}
\sum_{\substack{\alpha\in\mathcal{S}_{\beta} \\ \alpha\sim\gamma}}
\left(\frac{1-p_2}{p_2}\right)^{\left|\left\{i \in [R]:\ell_{i}\geqslant m-1,\alpha_{s_{i-1}+k_{i}+m-1}=1\right\}\right|}.
\end{align*}
Let $a = \left|\left\{i \in [R]:\ell_{i}\geqslant m-1\right\}\right|$
and, for $\alpha \in \mathcal{S}_\beta$, let
$b_\alpha = \left|\left\{i \in [R]:\ell_{i}\geqslant m-1,\alpha_{s_{i-1}+k_{i}+m-1}=1\right\}\right|$. Note that $b_\alpha \leqslant a$ for all $\alpha \in \mathcal{S}_\beta$. 
Now, splitting the sum over all possible values of $b_\alpha$, we get
\begin{align*}
&\sum_{\alpha \in\mathcal{S}_{\beta}}
\Prob[\alpha \rightarrow \beta] \pi(\alpha)
=\frac{p_1^{N_{1}(\beta)}(1-p_1)^{n-N_{1}(\beta)}}{Z_{n,m}}\sum_{\gamma\in\Sigma_{\beta}}
\left( \frac{p_1}{1-p_1} \right)^{N_{1}(\gamma)}
\sum_{t=0}^{a}
\sum_{\substack{\alpha\in\mathcal{S}_{\beta}\\ \alpha\sim\gamma, b_\alpha = t}} 
\left(\frac{1-p_2}{p_2}\right)^{t}\nonumber\\
={}&\frac{p_1^{N_{1}(\beta)}(1-p_1)^{n-N_{1}(\beta)}}{Z_{n,m}}\sum_{\gamma\in\Sigma_{\beta}}
\left( \frac{p_1}{1-p_1} \right)^{N_{1}(\gamma)}
\sum_{t=0}^{a} \binom{a}{ t} \left(\frac{1-p_2}{p_2}\right)^{t}.
\end{align*}
Using the binomial theorem, we get
\begin{align*}
\sum_{\alpha \in\mathcal{S}_{\beta}}
\Prob[\alpha \rightarrow \beta] \pi(\alpha)
=\frac{p_1^{N_{1}(\beta)}(1-p_1)^{n-N_{1}(\beta)}}{p_2^{a} \, Z_{n,m}}\sum_{\gamma\in\Sigma_{\beta}}
\left( \frac{p_1}{1-p_1} \right)^{N_{1}(\gamma)}.
\end{align*}
Recall that, for \emph{any} representative configuration $\gamma$, the only cells that \emph{may} be occupied by $1$ are those with indices $j\in\{s_{i-1}+k_{i}+m,\ldots,s_{i}-1\}$, for $i\in[R]$ with $\ell_{i}>(m-1)$. Therefore, 
$N_{1}(\gamma)$ ranges from $0$ to $c = \sum_{i \in [R]} (\ell_{i}-m+1)_+$.
Next, the number of representative configurations $\gamma$ for which we have $N_{1}(\gamma)=t'$ is equal to the number of ways that $t'$ cells may be selected out of these cells 
and this can be done in  $\binom{c}{t'}$ ways. 
Thus, replacing the sum over $\gamma$ by a sum over $t'$, we get
\begin{align}
\sum_{\alpha \in\mathcal{S}_{\beta}}
\Prob[\alpha \rightarrow \beta] \pi(\alpha)
=&\frac{p_1^{N_{1}(\beta)}(1-p_1)^{n-N_{1}(\beta)}}{p_2^{a} \, Z_{n,m}}
\sum_{t'=0}^{c} \sum_{\substack{\gamma\in\Sigma_{\beta} \\ N_{1}(\gamma)=t'} }
\left(\frac{p_1}{1-p_1}\right)^{t'}\nonumber\\
={}&\frac{p_1^{N_{1}(\beta)}(1-p_1)^{n-N_{1}(\beta)}}{p_2^{a}\, Z_{n,m}}
\sum_{t'=0}^{c} \binom{c}{t'}
\left(\frac{p_1}{1-p_1}\right)^{t'}\nonumber\\
={}&\frac{p_1^{N_{1}(\beta)}(1-p_1)^{n-c-N_{1}(\beta)}}{p_2^{a}\, Z_{n,m}},\label{left_side_master_equation_simplified}
\end{align}
where we have again used the binomial theorem in the final step.

The final task is to represent the exponents of the various terms in \eqref{left_side_master_equation_simplified} in terms of $N_{1}(\beta)$, $N_{10^{r}1}(\beta)$ for $r\in[m-2]$, and $N_{0^{m-1}1}(\beta)$. We begin by noting that, for each $i\in[R]$, 
\begin{align}
{}&\left(\beta_{s_{i-1}+k_{i}},\beta_{s_{i-1}+k_{i}+1},\ldots,\beta_{s_{i}},\beta_{s_{i}+1}\right)=10^{r}1 \text{ whenever } \ell_{i}=r \text{ for some } r\in[m-2],\nonumber\\
{}&\left(\beta_{s_{i}-m+2},\beta_{s_{i}-m+3},\ldots,\beta_{s_{i}},\beta_{s_{i}+1}\right)=0^{m-1}1 \text{ whenever } \ell_{i}\geqslant(m-1).\nonumber
\end{align}
Using these observations, the exponent of $p_2$ in \eqref{left_side_master_equation_simplified} can be written as $a=-N_{0^{m-1}1}(\beta)$, and the exponent of $(1-p_1)$ in \eqref{left_side_master_equation_simplified} can be written as
\begin{align*}
n-N_{1}(\beta)-c ={}&n-N_{1}(\beta) +(m-1)a
-\sum_{i\in[R]:\ell_{i}\geqslant(m-1)}\ell_{i}\nonumber\\
={}&n-N_{1}(\beta) +(m-1)N_{0^{m-1}1}(\beta) 
-\sum_{i\in[R]}\ell_{i}+\sum_{i\in[R]:\ell_{i}\leqslant(m-2)}\ell_{i}.
\end{align*}
Now, notice that $n-N_{1}(\beta)-\sum_{i\in[R]}\ell_{i} = 
n-\sum_{i\in[R]}k_{i}-\sum_{i\in[R]}\ell_{i} = 0$. 
Thus, this exponent can be written as
\[
\sum_{r\in[m-2]} r \left|\left\{i\in[R]:\ell_{i}=r\right\}\right|+(m-1)N_{0^{m-1}1}(\beta) 
=\sum_{r\in[m-2]}r N_{10^{r}1}(\beta)+(m-1)N_{0^{m-1}1}(\beta).
\]
Incorporating these calculations
into \eqref{left_side_master_equation_simplified}, we obtain the expression for $\beta$ as given by \eqref{limiting_distribution}, completing the proof when 
$\beta \neq 0^n$.
\end{proof}

\section{The partition function and density}
\label{sec:proof_main_2}

\begin{proof}[Proof of \cref{thm:main_2}]
Note, from \eqref{limiting_distribution} of \cref{thm:main_1}, that 
\begin{align*}
Z_{n,m} \equiv {}&Z_{n,m}(p_1,p_2)=\sum_{\beta \in \Omega}p_1^{N_{1}(\beta)}(1-p_1)^{\sum_{s\in[m-2]}sN_{10^{s}1}(\beta)+(m-1)N_{0^{m-1}1}(\beta)}p_2^{-N_{0^{m-1}1}(\beta)} \\
={}&1+\sum_{\beta \in \Omega\setminus\{0^{n}\}}p_1^{N_{1}(\beta)}(1-p_1)^{\sum_{s\in[m-2]}sN_{10^{s}1}(\beta)}\left(\frac{(1-p_1)^{m-1}}{p_2}\right)^{N_{0^{m-1}1}(\beta)}.
\end{align*}
We now refine the sum over $\beta$ as follows. Let $k$ be the number of $1$'s in $\beta$. 
We now consider all integer tuples $(M,N)$ such that
\begin{enumerate}[label=(O\arabic*), ref=O\arabic*]
\item \label{O1} $M=\sum_{s\in[m-2]}sN_{10^{s}1}(\beta)$ equals the total number of $0$s in $\beta$ accounted for by the occurrences of the subsequences $10^{s}1$, for all $s \in [m-2]$, in $\beta$, which cannot exceed the \emph{total} number of $0$s in $\beta$, which is $n-N_{1}(\beta)=(n-k)$;
\item \label{O3} and $N_{0^{m-1}1}(\beta)=N$, the total number of $0$s in $\beta$ due to the occurrences of the subsequence $0^{m-1}1$ in $\beta$ ought to be at least $(m-1)N$. 
\end{enumerate}
\eqref{O1} yields $0\leqslant M \leqslant (n-k)$ and
that the total number of $0$s in $\beta$ due to the occurrences of the subsequence $0^{m-1}1$ in $\beta$ equals $(n-k-M)$.
Thus, if $M<(n-k)$, the subsequence $0^{m-1}1$ must occur at least once in $\beta$, or, equivalently, $N\geqslant 1$, and the total number of $0$s in $\beta$ accounted for by the occurrences of the subsequence $0^{m-1}1$ in $\beta$ would be, by \eqref{O3}, at least $(m-1)$. Consequently, $M<(n-k) \implies M+(m-1)\leqslant(n-k) \implies M\leqslant (n-k-m+1)$. This is why, on one hand, we must have $M\in\{0,1,\ldots,n-k-m+1,n-k\}$, and on the other, we must have $N=0$ if and only if $M=(n-k)$. Finally, from \eqref{O1} and \eqref{O3}, we conclude that 
$(m-1)N\leqslant(n-k-M) \implies N\leqslant \lfloor(n-k-M)/(m-1)\rfloor$.
Therefore, we have shown that $(M,N)\in\mathcal{C}_{n,k}$, where
we recall the set $\mathcal{C}_{n,k}$ defined in \eqref{M_N_choices}. 
Thus, we can write
\begin{align}
Z_{n,m}={}&1+\sum_{k=1}^{n}\sum_{(M,N)\in\mathcal{C}_{n,k}}\sum_{\beta\in\Omega}p_1^{k}q_{1}^{M}q_{2}^{N}\chi\left[N_{1}(\beta)=k,\sum_{s\in[m-2]}sN_{10^{s}1}(\beta)=M,N_{0^{m-1}1}(\beta)=N\right],\label{Z_{n,m}_step_1}
\end{align}
where we set $q_{1}=(1-p_1)$ and $q_{2}=p_2^{-1}(1-p_1)^{m-1}$. 

\sloppy The coefficient for the term $p_1^{k}q_{1}^{M}q_{2}^{N}$ in \eqref{Z_{n,m}_step_1} is equal to the number of configurations $\beta \in \Omega$ such that $N_{1}(\beta)=k$, $\sum_{s\in[m-2]}sN_{10^{s}1}(\beta)=M$ and $N_{0^{m-1}1}(\beta)=N$, for each $k \in \{0,1,\ldots,n\}$ and each pair $(M,N)\in\mathcal{C}_{n,k}$. In what follows, this is the number that we count.

We focus on the subset of $\Omega$ comprising configurations $\beta$ with $N_{1}(\beta)=k$. Let us consider $k$ copies of the symbol $1$ placed around a circle, giving rise to $k$ gaps (each gap is between two consecutively placed $1$s). Let us mark one of these $k$ copies of $1$ as \emph{special} -- for the time being, we perform the counting keeping this special $1$ fixed at, say, the cell indexed $1$, but eventually, we shall allow this special $1$ to be placed in \emph{any one} of the $n$ cells and obtain the final count. 

Next, for any $M \in \{0,1,\ldots,n-k-m+1,n-k\}$, we fix $\overline{x}=(x_{1},x_{2},\ldots,x_{m-2})\in\mathcal{T}_{M}$, where $\mathcal{T}_{M}$ is as defined in \eqref{T_{M}}. Note that, depending on the value of $M$ being considered, the set $\mathcal{T}_{M}$ could be empty. Finally, we select $N\in \mathbb{N}_{0}$ such that $(m-1)N \leqslant (n-k-M)$ and $N=0$ if and only if $M=(n-k)$. 

Reminding the reader that we keep the special $1$ fixed in its position in cell $1$,  
\begin{enumerate}
\item we now select $x_{1}$ of the $k$ gaps (each gap, as mentioned above, is flanked on either side by a copy of $1$) available to us in $\binom{k }{x_{1}}$ ways, and in each of these chosen gaps, we place precisely one $0$,
\item we then select $x_{2}$ of the remaining $(k-x_{1})$ gaps in 
$\binom{k-x_{1}}{x_{2}}$ ways, and in each of these chosen gaps, we place precisely $2$ copies of $0$,
\item continue this way selecting $x_3, \dots, x_{m-3}$,
\item finally, we select $x_{m-2}$ of the remaining $\left(k-\sum_{s=1}^{m-3}x_{s}\right)$ gaps in $\binom{k - x_1 - \cdots - x_{m-3} }{x_{m-2}}$ ways, and in each of these chosen gaps, we place precisely $(m-2)$ copies of $0$.
\end{enumerate}
It is evident that, in a configuration constructed in such a manner, there are precisely $x_{s}$ occurrences of the subsequence $10^{s}1$, for each $s \in [m-2]$.

Following this, we select $N$ out of the remaining $\left(k-\sum_{s=1}^{m-2}x_{s}\right)$ gaps in $\binom{k - x_1 - \cdots - x_{m-2}}{N}$ ways, and our task, now, is to populate each of these $N$ chosen gaps with \emph{at least} $(m-1)$ copies of $0$. The number of $0$s present in the configuration constructed so far equals $M=\sum_{s\in[m-2]}s x_{s}$ (since $\overline{x}=(x_{1},x_{2},\ldots,x_{m-2})\in\mathcal{T}_{M}$). Therefore, the number of $0$s that must be present due to the occurrences of the subsequence $0^{m-1}1$, equals $(n-k-M)$. Consequently, our task is to place a total of $(n-k-M)$ copies of the symbol $0$ in the $N$ gaps chosen at the beginning of this paragraph, making sure that each gap contains at least $(m-1)$ copies. This is equivalent to the problem of counting the number of ordered tuples $(y_{1},y_{2},\ldots,y_{N})\in\mathbb{N}^{N}$ such that $\sum_{i \in [N]} y_{i}=n-k-M$ and $y_{i}\geqslant(m-1)$ for each $i\in[N]$. Equivalently, our task is to count the number of ordered tuples $(z_{1},z_{2},\ldots,z_{N})\in\mathbb{N}_{0}^{N}$ (where we set $z_{i}=y_{i}-(m-1)$ for each $i\in[N]$) such that $\sum_{i \in [N]} z_{i}=n-k-M-(m-1)N$, and it is well-known that this number equals $\binom{n-k-M-(m-2)N-1}{ N-1}$. Once we have placed $y_{1}$ copies of $0$ in the first of the $N$ gaps chosen above, $y_{2}$ copies in the second, and so on, we have precisely $N$ occurrences of the subsequence $0^{m-1}1$ in the configuration thus constructed.

Following the construction described so far, the total number of configurations $\beta \in \Omega$, in which there are 
\begin{enumerate}[label=(P\arabic*), ref=P\arabic*]
\item \label{P1} precisely $k$ occurrences of the symbol $1$, with one of them occupying the cell indexed $1$, 
\item \label{P2} precisely $x_{s}$ occurrences of the subsequence $10^{s}1$ for each $s \in [m-2]$, with $(x_{1},x_{2},\ldots,x_{m-2})\in\mathcal{T}_{M}$, 
\item \label{P3} and precisely $N$ occurrences of the subsequence $0^{m-1}1$, 
\end{enumerate}
equals 
\begin{equation}
\binom{k}{x_{1}} \binom{k-x_{1}}{x_{2}}
\cdots
\binom{k - x_1 - \cdots - x_{m-3}}{x_{m-2}}
\binom{k - x_1 - \cdots - x_{m-2}}{N} \binom{n-k-M-(m-2)N-1}{ N-1}.\nonumber
\end{equation}
When $N=0$, the last factor gives $\binom{-1 }{ -1}=1$ as stated before \cref{thm:main_2}.
So far, this computation was accomplished keeping the special $1$ fixed in cell $1$. We now allow the special $1$ to be moved around, placing it in each of the $n$ cells at a time, which brings in a factor of $n$. However, this also gives rise to some overcounting -- specifically, keeping the special $1$ fixed in cell $i_{1}$, if we obtain a configuration $\beta$, satisfying \eqref{P1}, \eqref{P2} and \eqref{P3}, such that the remaining $(k-1)$ copies of $1$ in $\beta$ are in cells $i_{2}, i_{3}, \ldots, i_{k}$, then this same configuration $\beta$ is also counted when the special $1$ occupies any one of the cells $i_{2}, i_{3}, \ldots, i_{k}$. Thus, each such configuration is counted precisely $k$ times as we move the special $1$ around over all the cells of $[n]$. Consequently, to eliminate this overcounting, we divide by the factor of $k$, making the total count of such configurations equal to
\begin{equation}
\frac{n}{k}
\binom{k}{x_{1}} \binom{k-x_{1}}{x_{2}}
\cdots \binom{k - x_1 - \cdots - x_{m-3}}{x_{m-2}}
\binom{k - x_1 - \cdots - x_{m-2}}{N}
\binom{n-k-M-(m-2)N-1}{ N-1}.\label{intermediate_partition_function}
\end{equation}
Summing over all $\overline{x}=(x_{1},x_{2},\ldots,x_{m-2})\in\mathcal{T}_{M}$, we obtain the coefficient of $p_1^{k}q_{1}^{M}q_{2}^{N}$ in \eqref{Z_{n,m}_step_1}. Finally, summing over all $(M,N)\in\mathcal{C}_{n,k}$ and over all $k \in\{1,\ldots,n\}$ (and adding a $1$ corresponding to $k=0$, which is equivalent to $\beta=0^{n}$), we obtain, for $Z_{n,m}$, the expression given by \eqref{Z_{n,m}_final}, completing the proof. 
\end{proof}

\begin{proof}[Proof of \cref{thm:limiting_prob_cell_occupied_by_1}]
This proof for general values of $m$ follows a very similar line of argument as that of \cref{thm:main_2}, where we again set $q_{1}=(1-p_1)$ and $q_{2}=p_2^{-1}(1-p_1)^{m-1}$. 
To find $\Prob[\eta_{1}=1]$, where $\eta=(\eta_{i} \mid i\in[n])$ follows the limiting distribution $\pi$ of the \pca{}, we have to sum the probabilities (under $\pi$) of all those configurations $\beta=(\beta_{i} \mid i\in[n])$ such that $\beta_{1}=1$. Recalling $\mathcal{C}_{n,k}$ as defined in \eqref{M_N_choices} and the form of $\pi(\beta)$ stated in \eqref{limiting_distribution}, we can write
\begin{align}
\Prob[\eta_{1}=1]={}&\sum_{\substack{\beta\in\Omega \\ \beta_{1}=1}} \pi(\beta) \nonumber\\
=& \sum_{k=1}^{n}\sum_{(M,N)\in\mathcal{C}_{n,k}}
\sum_{\substack{\beta\in\Omega \\ \beta_{1}=1}}
\chi\left[N_{1}(\beta)=k,\sum_{s\in[m-2]}s N_{10^{s}1}(\beta)=M, 
N_{0^{m-1}1}(\beta)=N \right]\pi(\beta)\nonumber\\
={}&\frac{1}{Z_{n,m}} \sum_{k=1}^{n} \sum_{(M,N)\in\mathcal{C}_{n,k}}
p_{1}^{k} q_{1}^{M}q_{2}^{N} \Bigg|\Bigg\{\beta\in\Omega \;\Bigg|\; \beta_{1}=1, 
N_{1}(\beta)=k,\sum_{s\in[m-2]}s N_{10^{s}1}(\beta)=M,\nonumber\\
& \qquad \qquad \qquad N_{0^{m-1}1}(\beta)=N \Bigg\}\Bigg|,\label{limiting_probability_cell_1_occupied_by_1}
\end{align}
\sloppy so that our task reduces to counting all those configurations $\beta\in\Omega$ such that $N_{1}(\beta)=k$, $\sum_{s\in[m-2]}s N_{10^{s}1}(\beta)=M$, $N_{0^{m-1}1}(\beta)=N$, and the cell $1$ is occupied by the symbol $1$ under $\beta$. Recalling $\mathcal{T}_{M}$ as defined in \eqref{T_{M}}, and fixing $(x_{1},x_{2},\ldots,x_{m-2})\in\mathcal{T}_{M}$, the number of configurations $\beta\in\Omega$ such that $N_{1}(\beta)=k$, $N_{10^{s}1}(\beta)=x_{s}$ for each $s\in[m-2]$, $N_{0^{m-1}1}(\beta)=N$ and $\beta_{1}=1$, can be counted in exactly the same way as we accomplished the counting needed in the proof of \cref{thm:main_2}, with one notable exception: to ensure $\beta_{1}=1$, we keep the `special $1$' fixed at site $1$, which is why the factor of $n/k$ is absent from \eqref{intermediate_partition_function}. 
Finally, incorporating these findings (i.e.\ the expression obtained from \eqref{intermediate_partition_function} after omission of the factor $n/k$) into \eqref{limiting_probability_cell_1_occupied_by_1}, we obtain:
\begin{align}
Z_{n,m}\Prob[\eta_{1}=1]={}&\sum_{k=1}^{n}\sum_{(M,N)\in\mathcal{C}_{n,k}}\sum_{(x_{1},x_{2},\ldots,x_{m-2})\in\mathcal{T}_{M}}\binom{k}{x_{1}}\binom{k-x_{1}}{x_{2}}\cdots
\binom{k - x_1 - \cdots - x_{m-3}}{x_{m-2}} \nonumber\\
& \times \binom{k - x_1 - \cdots - x_{m-2}}{N}
\binom{n-k-M-(m-2)N-1}{N-1}
p_{1}^{k} q_{1}^{M}q_{2}^{N},\nonumber
\end{align}
completing the proof. 
\end{proof}

\section{Neighbourhood size $m = 2$}
\label{sec:m=2}

We begin with a proof of the detailed balance condition.
	
\begin{proof}[Proof of \cref{prop:detailed m=2}]
It is easy to check that the stationary distribution in \eqref{limiting_distribution} satisfies the detailed balance condition if $n = 2$.
From now on, we assume $n > 2$. 

An important observation, following from \eqref{rule_1}, \eqref{rule_2} and \eqref{rule_3}, is that there is a transition from $\alpha$ to $\beta$ 
if and only if there is no position $i$ such that $\alpha_i = \beta_i = 1$. 
We shall use this fact repeatedly.
Therefore, for any pair of distinct configurations $\alpha, \beta \in \Omega$, there is a transition from $\alpha$ to $\beta$ if and only if there is a transition from $\beta$ to $\alpha$. 

First, suppose $p_1 + p_2 = 1$. Then, from \eqref{rule_1}, \eqref{rule_2} and \eqref{rule_3}, we see that every $0$ changes to a $1$ with probability $p_1$ without any dependence on the next site, and 
every $1$ changes to a $0$ with probability $1$. This is a classical chain which is known (and can easily be verified) to be reversible.

For the converse, we first simplify \eqref{limiting_distribution} for $m = 2$ by
introducing some notation. For any $\alpha \in \Omega$ and $a, b \in \{0, 1\}$, 
let $\Pos_a(\alpha)$ denote the set of positions of the letter $a$ in $\alpha$, and let $\Pos_{ab}(\alpha)$ denote the set of positions of the letter $a$ in $\alpha$ such that the subsequent letter is $b$. In other words,
\begin{equation}
\Pos_{a}(\alpha)=\left\{i\in[n] \mid \alpha_{i}=a\right\} \text{ and } \Pos_{ab}(\alpha)=\left\{i\in[n] \mid \alpha_{i}=a,\alpha_{i+1}=b\right\},\nonumber
\end{equation}
where the sum $(i+1)$ in the subscript is considered modulo $n$. Then we have
\begin{equation}
\label{pim2}
\pi(\alpha) = \frac{p_1^{|\Pos_{1}(\alpha)|}}{Z_{n,2}}
\left( \frac{1 - p_1}{p_2} \right)^{|\Pos_{01}(\alpha)|}.
\end{equation}
We now write down the formula for transition probabilities using this notation.
Suppose $\alpha, \beta \in \Omega$ and we are looking at $\Prob[\alpha \to \beta]$.
Let $\Pos_{a_1 b_1, a_2 b_2}(\alpha, \beta) \equiv  
\Pos_{a_1 b_1, a_2 b_2} = \Pos_{a_1 b_1}(\alpha) \cap \Pos_{a_2 b_2}(\beta)$
for $a_1, b_1, a_2, b_2 \in \{0, 1\}$. 
From \eqref{rule_1}, we get a factor of $p_1$ every time there is a consecutive pair $00$ in $\alpha$ and there is a $1$ in $\beta$ at the location of the first $0$ in that pair. 
This contributes a factor of $p_1^{|\Pos_{00,10}| + |\Pos_{00,11}|}$ towards this transition. Arguing similarly from \eqref{rule_2}, we get
\begin{equation}
\label{a to b}
\Prob[\alpha \to \beta] = p_1^{|\Pos_{00,10}| + |\Pos_{00,11}|}
(1-p_1)^{|\Pos_{00,00}| + |\Pos_{00,01}|}
p_2^{|\Pos_{01,00}|}
(1-p_2)^{|\Pos_{01,10}|},
\end{equation}
where we have used the fact that $|\Pos_{01,01}| = |\Pos_{01,11}| = \emptyset$, since $\alpha$ and $\beta$ cannot have $1$'s at the same location as argued above.
By the same argument, 
\begin{equation}
\label{b to a}
\Prob[\beta \to \alpha] = p_1^{|\Pos_{10,00}| + |\Pos_{11,00}|}
(1-p_1)^{|\Pos_{00,00}| + |\Pos_{01,00}|}
p_2^{|\Pos_{00,01}|}
(1-p_2)^{|\Pos_{10,01}|}.
\end{equation}
The ratio of the transitions \eqref{a to b} and \eqref{b to a} is
\begin{align}
\label{ratio trans}
\frac{\Prob[\alpha \to \beta]}{\Prob[\beta \to \alpha]} =& 
p_1^{|\Pos_{00,10}| + |\Pos_{00,11}| - |\Pos_{10,00}| - |\Pos_{11,00}|} \cr
& \times (1-p_2)^{|\Pos_{01,10}| - |\Pos_{10,01}|}
\left( \frac{1-p_1}{p_2} \right)^{|\Pos_{00,01}| - |\Pos_{01,00}|}.
\end{align}

Now that we have written down all the quantities in \eqref{detailed balance}, we
collect the various factors in the ratio $(\pi(\alpha) \Prob[\alpha \to \beta])/ (\pi(\beta) \Prob[\beta \to \alpha])$ using \eqref{pim2} and \eqref{ratio trans}. First, the power of $p_1$ is
\[
|\Pos_{1}(\alpha)| - |\Pos_{1}(\beta)| + |\Pos_{00,10}| + |\Pos_{00,11}| - |\Pos_{10,00}| - |\Pos_{11,00}|.
\]
\sloppy Now, notice that 
$|\Pos_{00,10}| + |\Pos_{00,11}| = |\Pos_{00}(\alpha) \cap \Pos_1(\beta)|$ 
and similarly, 
$|\Pos_{10,00}| + |\Pos_{11,00}| = |\Pos_{1}(\alpha) \cap \Pos_{00}(\beta)|$.
Therefore, this power is 
\[
|\Pos_{1}(\alpha)| - |\Pos_{1}(\beta)| +
|\Pos_{00}(\alpha) \cap \Pos_1(\beta)| - |\Pos_{1}(\alpha) \cap \Pos_{00}(\beta)|.
\]
Combining terms and using the fact that $\beta$ cannot have a $1$ at any location in $\Pos_1(\alpha)$, it is clear that 
\[
|\Pos_{1}(\alpha)| - |\Pos_{1}(\alpha) \cap \Pos_{00}(\beta)|
= |\Pos_{1}(\alpha) \cap \Pos_{01}(\beta)|.
\]
But this last term is equal to $|\Pos_{10}(\alpha) \cap \Pos_{01}(\beta)| = 
|\Pos_{10,01}|$, again by the same argument. Therefore, the power of $p_1$ at the end of the day turns out to be
\begin{equation}
\label{p1 power}
|\Pos_{10,01}| - |\Pos_{01,10}|.
\end{equation}

Happily, the factors of $1 - p_1$ and $p_2$ occur together in both \eqref{pim2} and \eqref{ratio trans}. Therefore, we can combine them and see that the power of $(1 - p_1)/p_2$ in the ratio is
\[
|\Pos_{00,01}| - |\Pos_{01,00}| + |\Pos_{01}(\alpha)| - |\Pos_{01}(\beta)|.
\]
But
\begin{multline*}
|\Pos_{01}(\alpha)| - |\Pos_{01,00}|
= |\Pos_{01}(\alpha)| - |\Pos_{01}(\alpha) \cap \Pos_{00}(\beta)| \\
= |\Pos_{01}(\alpha) \cap \Pos_{10}(\beta)| = |\Pos_{01,10}|,
\end{multline*}
where we have again used the fact that $\alpha$ and $\beta$ cannot have $1$'s at the same position. Repeating this argument for $\beta$, we get that the power 
of $(1 - p_1)/p_2$  is $\Pos_{01,10} - |\Pos_{10,01}|$, which is the negative of \eqref{p1 power}.

Finally, the power of $1 - p_2$ in the ratio is exactly given in \eqref{ratio trans} and that is also the negative of \eqref{p1 power}. Combining all these calculations, we have that
\[
\frac{\pi(\alpha) \Prob[\alpha \to \beta]}{\pi(\beta) \Prob[\beta \to \alpha]}
=
\left(\frac{p_1 p_2}{(1 - p_1)(1 - p_2)} \right)^{|\Pos_{10,01}| -|\Pos_{01,10}|}.
\]
For detailed balance to hold, the factor on the right hand should be $1$, which is satisfied if and only if $p_1 + p_2 = 1$.
\end{proof}

We now prove the formula for the generating function of the partition function.

\begin{proof}[Proof of \cref{thm:main_3}]
We shall prove here that the generating function corresponding to the sequence $\{Z_{n,2}\}$, with $Z_{n,2} \equiv Z_{n,2}(p_1,p_2)$, satisfies the identity given by \eqref{partition_function_gf}. From \cref{thm:main_2}, we have
\begin{equation}\label{generating_function_general}
Z_{n,2} = 
1 + \sum_{k=1}^{n} \sum_{j = 0}^{n-k} \frac{n}{k} \binom{k}{j}
\binom{n-k-1}{j-1} p_1^{k} q_{2}^{j},
\end{equation} 
for $n \geqslant 1$, where we recall, from the proof of \cref{thm:main_2} in \cref{sec:proof_main_2}, that $q_{2}=p_2^{-1}(1-p_1)$.
In \eqref{generating_function_general}, we focus on the latter term, i.e.\
\begin{equation}
\sum_{n = 1}^\infty \sum_{k=1}^{n} \sum_{j = 0}^{n-k} \frac{n}{k} \binom{k}{j}
\binom{n-k-1}{j-1} p_1^{k} q_{2}^{j} x^n.
\end{equation}
Since the limits are natural, we can move the $n$-sum inside, to obtain
\begin{equation}
\sum_{k=1}^{\infty} \frac{p_1^{k}}k \sum_{j = 0}^{\infty} \binom{k}{j} q_{2}^{j}
\sum_{n = j+k}^\infty n \binom{n-k-1}{j-1}  x^n.\label{generating_function_n_sum_inside}
\end{equation}
In order to perform the $n$-sum, we begin by noting that for any $j\in\mathbb{N}$, applying Newton's binomial expansion, we get
\begin{align}
(1-x)^{-j}={}&\sum_{i=0}^{\infty}{-j\choose i}(-x)^{i}=\sum_{i=0}^{\infty}\frac{(-j)(-j-1)\ldots(-j-i+1)}{i!}(-x)^{i}\nonumber\\
={}&\sum_{i=0}^{\infty}\frac{j(j+1)\ldots(j+i-1)}{i!}x^{i}=\sum_{i=0}^{\infty}\binom{j+i-1}{j-1}x^{i}.\label{Newton}
\end{align}
The $n$-sum of \eqref{generating_function_n_sum_inside} can now be performed as follows:
\begin{align}
{}&\sum_{n=j+k}^{\infty}n\binom{n-k-1}{j-1}x^{n}=x^{j+k}\sum_{t=0}^{\infty}(t+j+k)\binom{t+j-1}{j-1}x^{t}\nonumber\\ 
={}&x^{j+k}\sum_{t=0}^{\infty}(t+1)\binom{t+j-1}{j-1}x^{t}+x^{j+k}\sum_{t=0}^{\infty}(j+k-1)\binom{t+j-1}{j-1}x^{t}\nonumber\\ 
={}&x^{j+k}\frac{d}{dx}\left(\sum_{t=0}^{\infty}\binom{t+j-1}{j-1}x^{t+1}\right)+(j+k-1)x^{j+k}\sum_{t=0}^{\infty}\binom{t+j-1}{j-1}x^{t}\nonumber\\
={}&x^{j+k}\frac{d}{dx}\left\{x(1-x)^{-j}\right\}+(j+k-1)x^{j+k}(1-x)^{-j}, \quad \text{by \eqref{Newton};}\nonumber\\
={}&x^{j+k}(1-x)^{-j}+j x^{j+k+1}(1-x)^{-j-1}+(j+k-1)x^{j+k}(1-x)^{-j}\nonumber\\
={}&x^{j+k}(1-x)^{-j-1}\left\{(1-x)+j x+(j+k-1)(1-x)\right\}=\frac{x^{j+k} (j + k - k x)}{(1 - x)^{j+1}}.\label{n-sum}
\end{align}
Substituting the final expression obtained from \eqref{n-sum} into \eqref{generating_function_n_sum_inside}, we now perform the $j$-sum as follows:
\begin{align}
{}&\sum_{j = 0}^{\infty} \frac{x^{j+k} (j + k - k x)}{(1 - x)^{j+1}}
\binom{k}{j} q_{2}^{j} = x^{k}\sum_{j=0}^{k}\binom{k}{j}\frac{j x^{j} q_{2}^{j}}{(1-x)^{j+1}}+k x^{k}\sum_{j=0}^{k}\binom{k}{j}\frac{x^{j}q_{2}^{j}}{(1-x)^{j}}\nonumber\\
={}&\frac{k q_{2} x^{k+1}}{(1-x)^{2}}\sum_{j=1}^{k}\binom{k-1}{j-1}\left(\frac{x q_{2}}{1-x}\right)^{j-1}+k x^{k}\sum_{j=0}^{k}\binom{k}{j}\left(\frac{x q_{2}}{1-x}\right)^{j}\nonumber\\
={}&\frac{k q_{2} x^{k+1}}{(1-x)^{2}}\left(1+\frac{x q_{2}}{1-x}\right)^{k-1}+k x^{k}\left(1+\frac{x q_{2}}{1-x}\right)^{k}\nonumber\\
={}&\frac{k x^k (1-x+q_{2}x)^{k-1} (1 - 2x + 2q_{2}x + x^2 - q_{2}x^2)}{(1-x)^{k+1}}.\label{j-sum}
\end{align}
Finally, substituting the final expression from \eqref{j-sum} into \eqref{generating_function_n_sum_inside}, we perform the $k$-sum, which actually boils down to a geometric series, to obtain:
\begin{equation}
\sum_{k=1}^{\infty} \frac{p_1^{k}}k 
\frac{k x^k (1-x+q_{2}x)^{k-1} (1 - 2x + 2q_{2}x + x^2 - q_{2}x^2)}{(1-x)^{k+1}}
=
\frac{p_1x - 2p_1x^2 + 2p_1q_{2}x + p_1x^3 - p_1q_{2}x^3}
{(1 - x)(1 - x - p_1x + p_1x^2 - p_1q_{2}x^2)}.
\end{equation}
We now add the first term from $Z_{n,2}$ for $n \geqslant 1$, as well as $Z_{0,2}$, to
get
\begin{equation}
2 + \frac{x}{1 - x} + 
\frac{p_1x - 2p_1x^2 + 2p_1q_{2}x + p_1x^3 - p_1q_{2}x^3}
{(1 - x)(1 - x - p_1x + p_1x^2 - p_1q_{2}x^2)}
= \frac{2 - x - p_1x}{1 - x (1 + p_1) - x^2 p_1(q_{2} - 1)},
\end{equation}
as desired, thus completing the proof.
\end{proof}

\begin{proof}[Proof of \cref{thm:density m=2}]
Again, writing $q_{2}=p_{2}^{-1}(1-p_{1})$, the expression in \eqref{eq:limiting_prob_cell_occupied_by_1} boils down to
\begin{align}
Z_{n,2}\Prob[\eta_{1}=1]={}&\sum_{k=1}^{n}\sum_{j = 0}^{n-k}\binom{k}{j}
\binom{n-k-1}{j-1} p_1^{k} q_{2}^{j},\nonumber
\end{align}
so that the generating function corresponding to this sequence is given by:
\begin{align}
{}&\sum_{n\in\mathbb{N}}\left(\sum_{k=1}^{n}\sum_{j=0}^{n-k}\binom{k}{j}
\binom{n-k-1}{j-1} p_1^{k} q_{2}^{j}\right)x^{n}=\sum_{k=1}^{\infty}p_{1}^{k}\sum_{j=0}^{k}{k \choose j}q_{2}^{j}\sum_{n=k+j}^{\infty}{n-k-1 \choose j-1}x^{n}\nonumber\\
={}&\sum_{k=1}^{\infty}p_{1}^{k}\sum_{j=0}^{k}{k \choose j}q_{2}^{j}\sum_{i=0}^{\infty}{i+j-1 \choose j-1}x^{i+j+k}, \quad \text{setting } i=n-j-k;\nonumber\\
={}&\sum_{k=1}^{\infty}p_{1}^{k}\sum_{j=0}^{k}{k \choose j}q_{2}^{j}x^{j+k}\sum_{i=0}^{\infty}{i+j-1 \choose j-1}x^{i}=\sum_{k=1}^{\infty}p_{1}^{k}\sum_{j=0}^{k}{k \choose j}q_{2}^{j}x^{j+k}(1-x)^{-j}\nonumber\\
={}&\sum_{k=1}^{\infty}p_{1}^{k}x^{k}\sum_{j=0}^{k}{k \choose j}\left(\frac{q_{2}x}{1-x}\right)^{j}=\sum_{k=1}^{\infty}p_{1}^{k}x^{k}\left(1+\frac{q_{2}x}{1-x}\right)^{k}=\frac{p_{1}x(1-x+q_{2}x)}{(1-x)(1-p_{1}x)-p_{1}q_{2}x^{2}},\nonumber
\end{align}
which completes the proof.
\end{proof}

We now calculate the phase diagram of the \pca{} when the size of the neighbourhood-marking set equals $m = 2$. To begin with, we look at the asymptotics of the partition function.
From standard asymptotic analysis~\cite[Chapter V]{flajolet-sedgewick-2009}, it is well known that the growth of a sequence whose generating function is rational is exponential, and the growth rate is determined by the pole nearest to the origin.
In our problem, there are two cases to be considered, depending on whether $q_{2} = 1$ or not. 

The first, when $q_{2} = 1$, is particularly easy. In this case, the denominator in \eqref{partition_function_gf} is linear, and is given by $1 - (1+p_1)x$. 
Therefore, by a standard geometric series calculation, we get
\begin{equation}
\label{pf q=1}
Z_{n,2} = (1+2p_1) (1+p_1)^{n-1},
\end{equation}
which can also be obtained directly from \cref{thm:main_2}. In the second case, i.e.\ when $q_{2} \neq 1$, the poles of $\sum Z_{n,2} x^n$ occur at
\begin{equation}
x_{\pm} = \frac{-(1+p_1) \pm \sqrt{(1-p_1)^2 + 4 p_1q_{2}}}{2 p_1(q_{2}-1)}.
\end{equation}
Since each of $p_1$ and $q_{2}$ is strictly positive, the roots are both real and distinct. Thus, we need to compare the absolute values of $x_+$ and $x_-$ over the range $p_1 \in (0,1)$ and 
$q_{2} \in (0, \infty) \setminus \{1\}$. One can check that $|x_+| < |x_-|$ in this range. 
Now write
\begin{equation}
\sum Z_{n,2} x^n = \frac{\mathcal{Z}_0(x)}{(x_+ - x) (x_- - x)}
= \frac{\mathcal{Z}_0(x)}{x_+ (x_- - x)} \left( 1 - \frac{x}{x_+} \right)^{-1},
\end{equation}
where
\begin{equation}
\mathcal{Z}_0(x) = \frac{x + p_1x - 2}{p_1(q_{2}-1)}.
\end{equation}
Expanding the binomial coefficient and plugging in $x = x_+$ gives the asymptotic formula,
\begin{equation}
Z_{n,2} \sim \frac{\mathcal{Z}_0(x_+)}{x_+ (x_- - x_+)} x_+^{-n}.
\end{equation}

\section*{Acknowledgements}
\sloppy A.\ A.\ was partially supported by SERB Core grant CRG/2021/001592 as well as the DST FIST program - 2021 [TPN - 700661]. 
M.\ P.\ was partially supported by SERB Core grant CRG/2021/006785.

\bibliographystyle{alpha}
\bibliography{PCA_ring_bib}

\end{document}